\documentclass[10pt,reqno]{amsart}
\usepackage{amsmath,amssymb,latexsym,soul,cite,mathrsfs}
\usepackage{color,enumitem,graphicx}
\usepackage[colorlinks=true,urlcolor=blue,
citecolor=red,linkcolor=blue,linktocpage,pdfpagelabels,
bookmarksnumbered,bookmarksopen]{hyperref}
\usepackage[english]{babel}
\usepackage{enumitem}

\usepackage[left=1.8cm,right=1.8cm,top=2.1cm,bottom=2.1cm]{geometry}

\usepackage[hyperpageref]{backref}

\makeatletter
\newcommand{\leqnomode}{\tagsleft@true}
\newcommand{\reqnomode}{\tagsleft@false}
\makeatother

\date{}
\def\nd{\noindent}
\def\thend{\rule{3mm}{3mm}}

\newtheorem{theorem}{Theorem}[section]

\newtheorem{prop}{Proposition}[section]
\newtheorem{lemma}{Lemma}[section]
\newtheorem{rmk}{Remark}[section]
\newtheorem{exm}{Example}[section]

\newcommand{\R}{\mathbb{R}}
\begin{document}
	\title[Stein-Weiss problems via nonlinear Rayleigh quotient for concave-convex nonlinearities]{Stein-Weiss problems via nonlinear Rayleigh quotient for concave-convex nonlinearities}
	\vspace{1cm}
	%%%%%%%%%%%%%%%%%%%%%%%%%%%%%%%%%%%%%%%%%%%%%%%%%%%%%%%%%%%%%%%%%%%%%%%%
	
	\author{Edcarlos D. Silva}
	\address{Edcarlos D da Silva \newline  Universidade Federal de Goias, IME, Goi\^ania-GO, Brazil}
	\email{\tt edcarlos@ufg.br}

	\author{Marcos. L. M. Carvalho}
	\address{M. L. M. Carvalho \newline Universidade Federal de Goias, IME, Goi\^ania-GO, Brazil }
	\email{\tt marcos$\_$leandro$\_$carvalho@ufg.br}
	
		\author{Márcia S. B. A. Cardoso}
	\address{Márcia S. B. A. Cardoso \newline Intituto Federal de Goi\'as, Campus Goi\^anica, Goi\^ania-GO, Brazil }
	\email{\tt marcia.cardoso@ifg.edu.br
	}

	
	\thanks{The first author was partially supported by CNPq with grant 309026/2020-2. The second author was partially supported by IFG with grant 23373004607/2019-44. The third author was partially supported by CNPq with grant 316643/2021-1.}

	\begin{abstract}
		In the present work, we consider existence and multiplicity of positive solutions for nonlocal elliptic problems driven by the Stein-Weiss problem with concave-convex nonlinearities defined in the whole space $\mathbb{R}^N$. More precisely, we consider the following nonlocal elliptic problem: 
		\begin{equation*}
			- \Delta u + V(x)u =   \lambda a(x) |u|^{q-2} u + \displaystyle \int \limits_{\mathbb{R}^N}\frac{b(y)\vert u(y) \vert^p dy}{\vert x\vert^\alpha\vert x-y\vert^\mu \vert y\vert^\alpha} b(x)\vert u\vert^{p-2}u, \,\, \hbox{in}\ \mathbb{R}^N, \,\, u\in H^1(\mathbb{R}^N),
		\end{equation*}
		where $\lambda >0, \alpha \in (0,N), N\geq3,
		0<\mu<N, 0 <
		\mu + 2 \alpha < N$. Furthermore, we assume also that $V: \mathbb{R}^N \to \mathbb{R}$ is a bounded potential, $a \in{L}^r(\mathbb{R}^N), a > 0$ in $\mathbb{R}^N$ and
		$b\in{L}^{t}(\mathbb{R}^N), b>0$ in $\mathbb{R}^N$ for some specific $r, t > 1$. We assume also that $1\leq q<2$ and $2_{\alpha,\mu} < p<2_{\alpha,\mu}^*$ where $2_{\alpha ,\mu}=(2N-2\alpha-\mu)/N$ and $2_{\alpha,\mu}^*= (2N-2\alpha-\mu)/(N-2)$.
		Our main contribution is to find the largest $\lambda^* > 0$ in such way that our main problem admits at least two positive solutions for each $\lambda \in (0, \lambda^*)$. In order to do that we apply the nonlinear Rayleigh quotient together with the Nehari method. Moreover, we prove a Brezis-Lieb type Lemma and a regularity result taking into account our setting due to the potentials $a, b : \mathbb{R}^N \to \mathbb{R}$.
	\end{abstract}
	
	\maketitle
	
	\section{Introduction}
	
	In the present work, we consider existence and multiplicity of positive solutions for nonlocal elliptic problems for the Stein-Weiss type problem. More specifically, we shall consider the following nonlocal elliptic problem:
	\begin{equation} \label{PEdcarlos}
		\left\{\begin{array}{ll} 
			- \Delta u + V(x) u =   \lambda a(x) |u|^{q-2}u + \displaystyle \int \limits_{\mathbb{R}^N}\frac{b(y)\vert u(y) \vert^p dy}{\vert x\vert^\alpha\vert x-y\vert^\mu \vert y\vert^\alpha} b(x)\vert u\vert^{p-2}u,\ \  \hbox{in}\ \mathbb{R}^N,
			\\ 
			u\in H^1(\mathbb{R}^N),
		\end{array}
		\right.  \tag{$P_\lambda$}
	\end{equation}
	where $\lambda >0, \alpha \in (0,N), N\geq3,
	0<\mu<N, 0 < \mu + 2 \alpha < N$. Furthermore, we assume also that $a \in{L}^r(\mathbb{R}^N), a > 0$ in $\mathbb{R}^N$ and
	$b\in{L}^{t}(\mathbb{R}^N), b>0$ in $\mathbb{R}^N$ for some $r, t > 1$. Here we assume also that $1\leq q<2$ and $2_{\alpha,\mu} < p<2_{\alpha,\mu}^*$ where $2_{\alpha ,\mu}=(2N-2\alpha-\mu)/N$ and $2_{\alpha,\mu}^*= (2N-2\alpha-\mu)/(N-2)$. Later on, we shall consider the hypotheses on the potentials $a, b : \mathbb{R}^N \to \mathbb{R}$.
	
	It is important to stress that several works have been done in the last years considering the Stein-Weiss problem with different kind of nonlinearities. Here we refer the reader to the important works
	\cite{Carlos, minbo, TARUN2022,Beckner,Biswas,Biswas1} where several Stein-Weiss type problems are considered. Now, assuming that $\lambda= 0, b \equiv 1$, the Problem \eqref{PEdcarlos} becomes the Schr\"odinger-Newton problem involving a nonlinearity of the Stein-Weiss type in the following form:
	\begin{equation*}\label{JoseCarlos}
		- \Delta u + V(x)u = \displaystyle\frac{1}{\vert x\vert^{\alpha}}\left(\displaystyle\int_{\mathbb{R}^N}\frac{F(u(y)) }{\vert x-y\vert^\mu\vert y\vert^{\alpha}}dy\right)f(u(x))\ \  \hbox{in}\ \mathbb{R}^N, 	u\in H^1(\mathbb{R}^N)
	\end{equation*}
	where $F(t) = \int_0^t f(s)ds, t \in \mathbb{R}$ and $f: \mathbb{R}\to \mathbb{R}$ is a continuous function. In most of these works the function $f$ satisfies the well-known Ambrosetti–Rabinowitz condition which proves that any Palais-Smale sequence is bounded. Hence, using some arguments from the machinery of variational methods, the authors ensured many results on existence of positive solutions for the Stein-Weiss problem. On this subject we refer the reader also to \cite{HLSPONDERADA,Chen,chenap,Zang,Zhou,Su,Yuan}. In \cite{minbo}, the authors considered the following nonlocal problem with the upper critical Sobolev exponent given in the following form: 
	\begin{equation*}
		- \Delta u + u = \displaystyle\frac{1}{\vert x\vert^{\alpha}}\left(\displaystyle\int_{\mathbb{R}^N}\frac{\vert u(y) \vert^{2_{\alpha,\mu}^*} }{\vert x-y\vert^\mu\vert y\vert^{\alpha}}dy\right)\;\;\vert u\vert^{2_{\alpha,\mu}^*-2}u\ \  \hbox{in}\ \mathbb{R}^N, 	u\in H^1(\mathbb{R}^N).
	\end{equation*}
	In that work, the authors developed a nonlocal version for the concentration principle of compactness in order to investigate the existence of nonnegative solutions. At the same time, the authors considered also the following Stein-Weiss type problem:
	\begin{equation*}
		- \Delta u + V(x) u = \displaystyle\frac{1}{\vert x\vert^{\alpha}}\left(\displaystyle\int_{\mathbb{R}^N}\frac{\vert  u(y) \vert^p }{\vert x-y\vert^\mu\vert y\vert^{\alpha}}dy\right)\;\;\vert u\vert^{p-2}u\ \  \hbox{in}\ \mathbb{R}^N, 	u\in H^1(\mathbb{R}^N),
	\end{equation*}
	where $N\geq3,\;0<\mu<N,\; \alpha\geq0,\; 0<2\alpha+\mu < N$ and
	$2_{\alpha, \mu}<p<2^*_{\alpha, \mu}$. 
	In that work, the authors obtained existence, regularity and symmetry of solutions.
	
	It is worthwhile to mention that the weighted Hardy-Littlewood-Sobolev inequality and the Stein-Weiss type inequality are the key for the proof of existence and multiplicity of solutions for the Stein-Weiss problems. This kind of inequalities are established in \cite{Stein_weiss1958}. Nonlocal elliptic problems driven by the Stein-Weiss problem have been extensively studied in different settings. The main point here is to ensure some $L^p$ estimates for integral operators with singular kernel. This is fundamental in harmonic analysis which has many applications, see \cite{stein2}. For related results we refer the reader to \cite{folland,sawyer,minboaw}.
	
	Recall that  the Problem \eqref{PEdcarlos} with $\alpha= 0$ becomes the Choquard type problem which can be written in the following form:
	\begin{eqnarray}\label{choquard}
		- \Delta u + V(x)u =a(x)\lambda |u|^{q-2} u + \displaystyle\int_{\mathbb{R}^N}\frac{b(y)\vert u(y) \vert^p }{\vert x-y\vert^\mu}dy\;b(x)\vert u\vert^{p-2}u\ \ \hbox{in}\ \mathbb{R }^N, 	u\in H^1(\mathbb{R}^N).
	\end{eqnarray}
	The Choquard type problems have been widely considered in the last decades which has also several physical applications. For instance, assuming that $N=3, \mu=1, V\equiv 1,p=2,\lambda=0$ and $b=1$, the Problem \eqref{choquard} boils down to the Choquard-Pekar equation
	\begin{equation*}
		- \Delta u + u =\left(\displaystyle\frac{1}{\vert x\vert}*\vert u\vert^2\right)u\ \  \hbox{in}\ \mathbb{R}^3, 	u\in H^1(\mathbb{R}^3).
	\end{equation*}
	These kind of problems were introduced in \cite{pekar} assuming that $N =3, \lambda = 0, \alpha = 0$ and $\mu = 1$. The main objective for that work is to describe of a polaron at rest in quantum field theory. The same problem is also used in order to describe an electron trapped in its own hole. The main idea here is to use an approximation to Hartree-Fock theory of one component plasma. Here we refer the interested reader to \cite{Penrose} where was also proposed a model of self-gravitating matter in which the reduction of the quantum state is considered a gravitational phenomenon. Now, this context is known as the nonlinear Schr\"odinger-Newton equation.

	In \cite{MOROZ},  the authors considered the following nonlocal elliptic problem: 
	\begin{equation*}
		-\Delta u+u=\left(I_{\alpha}*\vert u\vert^p\right)\vert u\vert^{p-2}u\;\;\hbox{in}\;\;\mathbb R^N, 	u\in H^1(\mathbb{R}^N),
	\end{equation*}
	where $I_{\alpha}$ is the Riesz potential and
	$p>1$. In that work, the authors ensured the existence of a positive ground state solution, that is, it was found at least one solution which has the lowest energy among any other nontrivial solution. In the same work, the author also established regularity results and positivity of ground states solutions as well as the existence of radially symmetric solutions. Later, in \cite{Moroz2}, the authors proved existence of a non-trivial solution $u \in H^1(\mathbb R^N)$ for the following nonlinear Choquard problem:
	\begin{equation*}
		-\Delta u+u=\left(I_{\alpha}*F(u)\right)f(u)\;\;\hbox{in}\;\;\mathbb R^N, 	u\in H^1(\mathbb{R}^N),
	\end{equation*}
	where $N\geq 3,\;\alpha\in (0,N),\; I_{\alpha}:\mathbb R^N\to \mathbb R$ is the Riesz potential. Under these conditions, the authors considered some existence results in the spirit of Berestycki-Lions for the nonlinearity $F$. 
	
	It is also important to stress that for each solution
	$u$ of Problem \eqref{PEdcarlos} the wave function $\Phi:\mathbb R \times \mathbb R^N \to \mathbb{C}$ defined by $\Phi(t,x)=e ^{it}u(x)$ is a solitary wave of the time-dependent Hartree equation. In fact, the function $\Phi$ satisfies the following nonlocal elliptic problem:
	$$-i\Phi_t-\Delta\Phi+W(x)\Phi-\left(\displaystyle\int_{\mathbb R^N}\displaystyle\frac{b(y)\vert \Phi\vert ^pdy}{\vert x\vert^{\alpha}\vert x-y\vert^{\mu}|y|^\alpha} \right)b(x)\vert \Phi\vert^{p-2}\Phi-\lambda a(x)\vert \Phi\vert^{q-2}\Phi=0,\;\;\hbox{in}\;\; \mathbb R\times\mathbb R^N,$$  
	where $W(x) = V(x)-1, x\in \mathbb R^N$ and $i$ is the imaginary unit. As a consequence, the Problem \eqref{PEdcarlos} can be understood as the stationary nonlinear Hartree equation.

	It is worthwhile to mention that many mathematical or physical applications on existence and multiplicity of solutions have been treated for local elliptic problems in recent decades. In the pioneer work \cite{Amb}, it was explored the sub-supersolution method and the mountain pass theorem proving the existence of a global solution. In that work, it was proved also a multiplicity result which depends on the size of $\lambda > 0$. More specifically, the author considered the following elliptic problem:
	\begin{eqnarray}\label{Amb} \ \ 
		\left\{\begin{array}{ll} 
			- \Delta u =   \lambda a(x) \vert u\vert ^{q-2} u + b(x)\vert u\vert^{p-2}u\ \  \hbox{in}\ \Omega,
			\\ 
			u>0\;\;\hbox{in}\;\; \Omega,\;\;u=0\;\;\hbox{on}\;\;\partial\Omega,
		\end{array}
		\right. 
	\end{eqnarray}
	where $\Omega\subset \mathbb R^N$ is a smooth bounded domain, $a = b = 1, 1 <q < 2 < p < 2^*$ and $2^* = 2N/(N -2)$ is the standard critical Sobolev exponent. For this problem the authors proved existence of at least two positive solutions for $0<\lambda<\lambda_0$ where $\lambda_0 > 0$. Furthermore, the authors proved that Problem \eqref{Amb} has at least one solution for $\lambda=\lambda_0$ and Problem \eqref{Amb} does not admit any weak solution for each $\lambda>\lambda_0$. Later on, considering some more general operators and different assumptions on the nonlinearity, existence and multiplicity of solutions for the Problem \eqref{Amb} was generalized in many directions. For instance, in \cite{Figueredo, Figueredo2}, it was considered some results on existence and multiplicity of solutions for local semilinear or local quasilinear elliptic problems defined in bounded domains. Furthermore, considering a potential $V:\mathbb R^N\to \mathbb R$ and $a= b =1$, the Problem \eqref{PEdcarlos} becomes the following Choquard-type equation
	\begin{eqnarray}\label{marcos e Edcarlos}
		- \Delta u +V(x) u =\lambda |u|^{q-2} u + \displaystyle\int \limits_{\mathbb{R}^N}\frac{\vert  u(y) \vert^p dy}{\vert x-y\vert^\mu} \vert u\vert^{p-2}u\ \  \hbox{in}\ \mathbb{R}^N, u \in H^1(\mathbb{R}^N).
	\end{eqnarray}
	This kind of problem appears in many works considering different assumptions on $V, p$ and $q$. For instance, see the important works \cite{Alves2,Alves3,ding,Gao3,RAY2021,Moroz_Van}. In \cite{RAY2021}, the authors considered the existence of ground and bound states solutions for the Problem \eqref{marcos e Edcarlos} for concave-convex nonlinearities. In that case, the authors looking for the parameter $\lambda>0$ obtaining some existence and multiplicity of solutions. In that work it was also considered the
	nonlinear Rayleigh quotient showing that there exists $\lambda^*>0$ in such way that the Nehari method can be applied for each $\lambda\in(0, \lambda^*)$. It is important to stress that the nonlinear Rayleigh quotient was developed in \cite{YAVDAT2017}. This method was studied also
	in recent years by many authors, see \cite{MYC,YAVDAT2005}. The main difficulty for these kind of problems is to guarantee
	that there exists an extreme value $\lambda^*>0$ such that the fibering maps does not admit any inflection point. Once again, the analysis is done combining the fibering map together with the Nehari method, see for instance \cite{metodoNehari}. The basic idea here is to ensure that fibering map associated to the energy functional admits at least two critical points for each $\lambda\in (0,\lambda^*)$.

	It is important to mention that for our main Problem \eqref{PEdcarlos} the nonlinear Rayleigh quotient can be applied depending on the size of $\lambda > 0$. The first difficulty arises from the fact that our potential $V$ is bounded from below and above by positive constants. Hence, some standard compact embedding from the Sobolev spaces into the Lebesgue space does not work anymore. Since we are looking for the Problem \eqref{PEdcarlos} with the lack of compactness the main ingredient is to write $\mathbb R^N=B(0,R)\cup \left(\mathbb R^N\setminus B(0,R)\right)$. It is also important to stress that $ H^1(B(0,R))\hookrightarrow L^{\gamma}(B(0,R))$ is a continuous embedding for each $\gamma\in[1,2^*]$. Moreover, we observe also that $H^1(B(0,R))\hookrightarrow\hookrightarrow L^{\gamma}(B(0,R))$ is a compact embedding for each $\gamma\in[1,2^*)$. However, for the set $\mathbb R^N\setminus B(0,R)$ the continuous and compact embedding are not available in general. Under these conditions, we shall use the potentials $a,b: \mathbb{R}^N \to \mathbb{R}$ in order to restore some kind of compactness. These can be done thanks to a Brezis-Lieb type identity which allows us to prove that the energy functional for our main problem is well-defined. Furthermore, in our setting due to the Brezis-Lieb type identity, we prove that the energy functional is in $C^1$ class. On the other hand, we prove a regularity result for our main problem which allow us to find existence and multiplicity of positive solutions. In fact, by using the Strong Maximum Principle, we ensure that our main problem admits at least two positive solutions. Once again, the Stein-Weiss problem give us some difficulties in order to control the right hand side of the Problem \eqref{PEdcarlos}. The main aim here is to apply the Brezis-Kato Theorem showing that any weak solution to the elliptic Problem \eqref{PEdcarlos} is sufficient smooth. Indeed, we need to control the nonlocal term using some fine estimates together with the H\"older inequality.  As was previously mentioned, we are concerned with the existence of ground and bound states solutions to the Problem \eqref{PEdcarlos} involving concave-convex nonlinearities. Hence, we need to control the parameter $\lambda > 0$ in order to apply the Lagrange Multiplier Theorem. Here we observe that the fibering map associated to the energy functional has inflection point for $\lambda = \lambda^*$ where $\lambda^*$ is an extreme value provided by the nonlinear Rayleigh quotient. Therefore, assuming that $\lambda \in (0, \lambda^*)$, we are able to prove that the fibering maps does not admit any inflection points. As far as we know, the present work is the first one proving the existence and multiplicity of positive solutions for the Stein-Weiss problem for concave-convex nonlinearities using the Rayleigh quotient. Hence, we are able to consider the largest $\lambda^* > 0$ such that for each $\lambda \in (0, \lambda^*)$ our main problem has at least two positive solutions.  
	
	\subsection{Assumptions and main theorems}
	
	As was mentioned previously, we are concerned with the existence of ground and bound states solutions to the Problem \eqref{PEdcarlos}. Recall that our main problem involves concave-convex nonlinearities of the Stein-Weiss type. Under these conditions, we need to control the parameter $\lambda > 0$ in order to avoid inflections points for the fibering map associated to the energy functional for the Problem \eqref{PEdcarlos}.  To overcome this difficulty, we shall consider the nonlinear Rayleigh quotient showing that there exists $\lambda^*> 0$ such that the Nehari method can be applied for each $\lambda\in (0,\lambda^*)$.  Throughout this work, we shall assume the following hypotheses:\\
	\\	
	\nd$(H_1)$ There exists  $V_0, V_\infty > 0$ such that $0 < V_0 \leq V(x) \leq V_\infty, x \in \mathbb{R}^N$. Assume also that $\lambda >0, \, \alpha\in(0,N), \,\ N\geq3, \,\, 0<2\alpha+\mu< N.$
	\\
	$(H_2)$ Suppose that $a > 0$ and $b>0$ in $\mathbb{R}^N$. Assume also that $1\leq q<2$, \, $2_{\alpha,\mu} < p<2_{\alpha,\mu}^*$ where
	$$\displaystyle 2_{\alpha,\mu}=\frac{2N-2\alpha-\mu}{N}\;\;\;\hbox{and}\;\;\;
	\displaystyle 2_{\alpha,\mu}^*=\frac{2N-2\alpha-\mu}{N-2}.$$\\
	$(H_3)$ 
	It holds $ a\in L^r(\mathbb R^N), \, \, \displaystyle b\in{L}^{\sigma}(\mathbb{R}^N) \cap L^{\sigma + \eta}(\mathbb{R}^N)$ where $\eta  > 0$ is small enough with $2^*= 2N/(N-2)$ and
	$$\displaystyle r =\frac{2^*}{2^*-q}\;\; \hbox{and}\;\;
	\sigma =\frac{2N}{2N-2\alpha-\mu-p(N-2)} = \dfrac{2^*}{2^*_{\alpha, \mu} - p}.$$
	$(H_4)$ It holds $a\in L^{N/2}(\mathbb R^N),\;\;b\in L^{\beta}(\mathbb R^N)\cap L^{\gamma}(\mathbb R^N)$ where
	$$\beta> \displaystyle\frac{2N}{4-(N-2)(p-2)}, \gamma = \frac{N\beta}{\beta[2-(N-2)(p-1)+(N-2\alpha-\mu)]-N} \,\,\ \mbox{and} \,\, p > \frac{2(N - 2 \alpha - \mu)}{N-2} - \frac{2^*}{\beta}.$$
	Now, we shall consider some examples for our setting. Namely, we consider the following examples:
	\begin{exm}
		Consider the potential $a: \mathbb{R}^N \to \mathbb{R}$ given by $a(x)= 1/(1+\vert x\vert^2)^{\gamma_{1}}, x \in \mathbb{R}^N$. It is not hard to verify that $a\in L^r(\R^N)\cap L^{N/2}(\R^N)$ holds true for each $\gamma_1 > (2N-q(N-2))/4$. In fact, by using the Co-area Theorem, we obtain the following estimates:
		\begin{eqnarray*}
			\displaystyle\int_{\R^N}\left(\displaystyle\frac{1}{\left(1+\vert x\vert^2\right)^{\gamma_1}}\right)^rdx &=& c\displaystyle\int_{0}^{+\infty}\displaystyle\frac{\theta^{N-1}}{(1+\theta^2)^{\gamma_1r}}d\theta \leq c_1+c\displaystyle\int_{1}^{+\infty}\displaystyle\frac{\theta^{N-1}}{(\theta^2)^{\gamma_1r}}d\theta = c_1+c\left.\displaystyle\frac{\theta^{N-2\gamma_1r}}{N-2\gamma_1r}\right|^{+\infty}_1<\infty.
		\end{eqnarray*}
	\end{exm}
	\begin{exm} Consider the potential $b: \mathbb{R}^N \to \mathbb{R}$ given by $b= 1/(1+\vert x\vert^2)^{\gamma_2}$. Assume also that $$\gamma=\frac{N\beta}{\beta[2-(N-2)(p-1)+ (N-2\alpha-\mu)]-N}.$$ It is easy to verify that $b\in L^{\sigma}(\R^N)\cap L^{\beta}(\R^N)\cap L^{\gamma}(\mathbb R^N)$ holds for each $$\gamma_2>\max\left\{\displaystyle\frac{2N-2\alpha-\mu-p(N-2)}{4},\displaystyle\frac{N}{2\gamma},\displaystyle\frac{N}{2\beta}\right\}.$$
		Indeed, for $\tau= \sigma,$ or $\tau=\beta,$ or $\tau=\varrho,$ and taking into account  the Co-area Theorem, we infer that		
		\begin{eqnarray*}
			\displaystyle\int_{\R^N}\left(\displaystyle\frac{1}{\left(1+\vert x\vert^2\right)^{\gamma_2}}\right)^\tau dx &=& c\displaystyle\int_{0}^{+\infty}\displaystyle\frac{\theta^{N-1}}{(1+\theta^2)^{\gamma_2\tau}}d\theta \leq  c_1+c\displaystyle\int_{1}^{+\infty}\displaystyle\frac{\theta^{N-1}}{(\theta^2)^{\gamma_2\tau}}d\theta =c_1+c\left.\displaystyle\frac{\theta^{N-2\gamma_2\tau}}{N-2\gamma_2\tau}\right|^{+\infty}_1<\infty.
		\end{eqnarray*}
		In particular, we see that $b\in L^{\tau}(\R^N)$.
	\end{exm}
	Now, we are looking for the working space for our main problem. Indeed, we consider the usual Sobolev space $H^1(\mathbb R^N)$ which is endowed with 
	the inner product and the norm given by 
	$$
	\left<u,\varphi\right>=\displaystyle\int\limits_{\mathbb R^N}(\nabla u\nabla \varphi+ V(x) u\varphi) dx,  \,\,	\Vert u\Vert=
	\left(\displaystyle\int_{\mathbb R^N}\vert \nabla u\vert^2+ V(x) \vert u\vert^2 dx\right)^\frac{1}{2},\;\;u, \phi \in H^1(\mathbb R^N).
	$$
	Furthermore, we consider some continuous and compact embedding, see for instance \cite[Theorem 1.2.1, Theorem 1.2.2]{BADIALE}. More specifically, we can state the following result:
	\begin{prop}\label{imersaomega}
		Let $\Omega\subset\mathbb R^N$ be a bounded domain and $N\geq3$. Then $H^1(\Omega)\hookrightarrow L^{\sigma}(\Omega)$ holds for each  $1\leq \sigma\leq2^*$.
		Furthermore, the previous embedding is compact if and only if
		$1 \leq \sigma <2^*$ where $2^*= 2N/)N-2)$ is the so called critical Sobolev exponent.
	\end{prop}
	\begin{rmk}\label{imersaoRN}
		Let $N\geq 3$ be fixed. It is important to emphasize that the embeddings $H^1(\mathbb R^N)\hookrightarrow L^{\sigma}(\mathbb R^N)$ are only continuous for each $\sigma\in [1,2^*]$. In other words, for each $\sigma\in [1,2^*]$, there exists  $S_\sigma > 0$ which does not depend on $u$ such that
		$\Vert u\Vert_{\sigma}\leq S_{\sigma}\Vert u\Vert$ for all $u\in H^1(\mathbb R^N ).$
	\end{rmk}
	
	Now, we mention that the energy functional $J_{\lambda}:H^1(\mathbb{R}^N) \longrightarrow \mathbb R$ is associated to the Problem \eqref{PEdcarlos} which is defined by
	\begin{equation}\label{FE}
		\displaystyle J_{\lambda}(u)=\frac{1}{2}{\Vert u \Vert}^2-\frac{\lambda}{q}\displaystyle\int_{\mathbb{R}^N}a(x){\vert u(x) \vert}^q dx
		-\frac{1}{2p}\displaystyle\int_{\mathbb R^N}\int_{\mathbb R^N}\frac{b(y)\vert u(y)\vert^pb(x)\vert u(x)\vert^p }{\vert x\vert^\alpha\vert x-y\vert^\mu \vert y\vert^\alpha}dxdy.
	\end{equation}
	For an easy reference we shall consider the following notations:
	\begin{equation}\label{A(u)}
		A(u)=\displaystyle\int_{\mathbb{R}^N}a(x){\vert u(x) \vert}^q dx,\;\;u\in H^1(\mathbb R^N),
	\end{equation}
	and
	\begin{equation}\label{B(u)}  B(u)=\displaystyle\displaystyle\int_{\mathbb R^N}\int_{\mathbb R^N}\frac{b(y)\vert u(y)\vert^pb(x)\vert u(x)\vert^p }{\vert x\vert^\alpha\vert x-y\vert^\mu \vert y\vert^\alpha}dxdy, \;\; u\in H^1(\mathbb R^N).
	\end{equation}
	Furthermore, we shall consider the following definitions:
	\begin{equation*}%\label{H(u)}
		H(u,\varphi)=\displaystyle\int\limits_{\mathbb R^N}a(x)\vert u\vert^{q-2}u\varphi\;dx,\;\;u,\varphi\in H^1(\mathbb R^N),
	\end{equation*}
	\begin{equation*}%\label{D(u)}
		D(u,\varphi)=\displaystyle\int\limits_{\mathbb R^N}\int\limits_{\mathbb R^N}\frac{b(y)\vert u(y)\vert^p}{\vert x\vert^\alpha\vert x-y\vert^\mu \vert y\vert^\alpha}b(x)\vert u(x)\vert^{p-2}u\varphi\;dx dy,\;\;u,\varphi\in H^1(\mathbb R^N).
	\end{equation*}
	
	It is important to stress that $u \in H^1(\mathbb R^N)$ is said to be a weak solutions for the Problem \eqref{PEdcarlos} whenever
	\begin{equation*}%\label{2}
		\displaystyle\int\limits_{\mathbb R^N}(\nabla u\nabla \varphi+u\varphi) dx= \lambda H(u,\varphi)+ D(u,\varphi),
	\end{equation*}
	Under these conditions, finding a weak solution $u \in H^1(\mathbb R^N)$  for the Problem \eqref{PEdcarlos} is equivalent to find critical points for the functional $J_\lambda$. Assuming hypotheses $(H_1)-(H_3)$ the energy functional is in $C^1$ class. Furthermore, the Gauteaux derivative of $J_\lambda$ is given by
	\begin{eqnarray}\label{J'}
		J_{\lambda}'(u)\varphi  & = & \langle u,\varphi\rangle-\lambda H(u,\varphi)-D(u,\varphi), \varphi \in H^1(\mathbb{R}^N).
	\end{eqnarray}
	Now, we shall consider the Nehari manifold as follows:
	\begin{equation}\label{Nehari}
		\displaystyle \mathcal {N}_{\lambda}= \left\{u \in {H^1}{(\mathbb R^N)}\setminus\{0\},  J_{\lambda}'(u)(u) = 0 \right\}.
	\end{equation}
	It is not difficult to verify that the functional $u \mapsto J_{\lambda}''(u)(u,u)$ is well defined for each $u\in H^1(\mathbb R^N)$. Furthermore, we infer that 
	\begin{equation}\label{J''}
		J_{\lambda}''(u)(u,u) = \Vert u\Vert^2-\lambda(q-1)A(u)- (2p-1)B(u).
	\end{equation}
	Hence, the functional $u \mapsto J_{\lambda}''(u)(u,u)$ is in $C^1$ class. The Nehari manifold can be splitted  as follows:
	\begin{equation*}\label{Nehari^+}
		\displaystyle\mathcal {N}_{\lambda}^+= \left\{u \in \mathcal {N}_{\lambda}\;\big |\; J_{\lambda}''(u)(u,u) > 0 \right\},
	\end{equation*}
	\begin{equation*}\label{Nehari^-}
		\displaystyle \mathcal {N}_{\lambda}^-=\left\{u \in \mathcal {N}_{\lambda} \;\big |\; J_{\lambda}''(u)(u,u) < 0 \right\},
	\end{equation*}
	\begin{equation*}\label{Nehari^0}
		\displaystyle \mathcal {N}_{\lambda}^0= \left\{u \in \mathcal {N}_{\lambda}\;\big |\; J_{\lambda}''(u)(u,u) = 0 \right\}.
	\end{equation*}
	The main objective in this work is to find solutions for  the following minimization problems:
	\begin{equation*}\label{C_+}
		C_{\mathcal {N}_{\lambda}^+}:=\inf\{J_{\lambda}(u)\;\big |\; u\in \mathcal {N}_{\lambda}^+\},
		C_{\mathcal {N}_{\lambda}^-}:=\inf\{J_{\lambda}(u)\;\big |\; u\in \mathcal {N}_{\lambda}^-\}.
	\end{equation*}
	In order to do that we consider the functionals $R_n;R_e : H^1(\mathbb R^N)\backslash \{0\}\rightarrow \mathbb R$ of $C^1$ class in the following form
	\begin{equation}\label{R_n}
		R_n(u)=\frac{\Vert u\Vert^2-B(u)}{A(u)},
		R_e(u)=\displaystyle\frac{\displaystyle\frac{1}{2}\Vert u\Vert^2-\displaystyle\frac{1}{2p}B(u)}{\displaystyle\frac{1}{q}A(u)}.
	\end{equation}
	Hence, we obtain that $R_n(u) = \lambda$ if and only if $u \in \mathcal{N}_\lambda$. Analogously, $R_e(u) = \lambda$ if and only if $u \in \mathcal{E}$ where $\mathcal{E} = \{ u \in H^1(\mathbb{R}^N) \setminus \{0\} : J_\lambda(u) = 0 \}$. Therefore, we can consider the following extremes:
	\begin{equation*}\label{lambda^*}
		\lambda^*:=\inf_{u \in H^1(\mathbb R^N \backslash\{0\})}\max_{t>0}R_n(tu),
		\lambda_*:=\inf_{u \in H^1(\mathbb R^N) \backslash\{0\}}\max_{t>0}R_e(tu).
	\end{equation*}
	Now, we are stay in position to state our main result as follows:
	\begin{theorem}\label{teorema} 
		Suppose $(H_1)-(H_3)$. Then we obtain that $0<\lambda_*<\lambda^*<\infty$ and the Problem \eqref{PEdcarlos} admits at least two nonegative solutions  $u_{\lambda},v_{\lambda} \in H^1(\mathbb R^N)$ for each $\lambda\in (0,\lambda^*)$. Moreover, assuming also that $(H_4)$ holds, we obtain that $u_\lambda, v_\lambda > 0$ in $\mathbb{R}^N$. Furthermore, the solutions $u_{\lambda},v_{\lambda}$ satisfy the following assertions:
		\begin{itemize}
			\item [$a)$] $J''_{\lambda}(u_{\lambda})(u_{\lambda},u_{\lambda})>0$ and $J''_{\lambda}(v_{\lambda})(v_\lambda, v_\lambda)<0.$
			\item [$b)$] The function $u_{\lambda}$ is a ground state solution  and $J_{\lambda}(u_{\lambda})<0.$
			\item[c)] The function $v_{\lambda}$ is a bound state solution which has the following properties:
			\begin{itemize}
				\item[$i)$] For each $\lambda\in(0,\lambda_*)$, we obtain that $J_{\lambda}(v_{\lambda})>0.$
				\item[$ii)$] For $\lambda=\lambda_*$ it holds $J_{\lambda}(v_{\lambda})=0.$
				\item[$iii)$] For each $\lambda\in(\lambda_*,\lambda^*)$ there holds $J_{\lambda}(v_{\lambda})<0.$
			\end{itemize}
		\end{itemize}
	\end{theorem}
	
	\begin{rmk}
		The hypothesis $(H_4)$ is used in order to prove that any weak solution $u \in H^1(\mathbb{R}^N)$ for the Problem \eqref{PEdcarlos} is in $W^{2,\theta}_{loc}(\mathbb R^N)\cap C^{1,\sigma}_{loc}(\mathbb R^N)$ for all $\theta \in (1,\infty)$ and for some $\sigma \in (0,1)$, see Proposition \ref{regular} ahead. The main idea here is to employ the Brezis-Kato Theorem, see \cite[Lemma B.3]{STRUWE}. Hence, by using the Strong Maximum Principle, we ensure that the weak solutions provided by Theorem \ref{teorema} satisfy  $u_\lambda > 0$ and $v_\lambda > 0$ in $\mathbb{R}^N$. In order to do that we consider the following assumption 
		$$ p > \frac{2(N - 2 \alpha - \mu)}{N-2} - \frac{2^*}{\beta}, N \geq 3.$$
		The last inequality is satisfied for $\beta > 0$ large enough and $\mu \geq 2, \alpha > 0$ with $0 < 2 \alpha + \mu < N, N \geq 3$. 
		
	\end{rmk}
	
	\subsection{Notation} Throughout this work we shall use the following notation:
	\begin{itemize}
		\item The characteristic function for the set $A \subset \mathbb{R}^N$ is denoted by $\chi_A$.
		\item The norm in $L^{r}(\mathbb{R}^N)$ and $L^{\infty}(\mathbb{R}^N)$, will be denoted respectively by $\|\cdot\|_{r}$ and $\|\cdot\|_{\infty},  r \in [1, \infty)$.
		\item $S_r$ denotes the best constant for the embedding $H^1(\mathbb{R}^N)\hookrightarrow L^r(\mathbb{R}^N)$ for each $r \in [2, 2^*]$.
		\item $B(x_0, R)$ denotes the open ball centered at $x_0 \in \mathbb{R}^N$ and radius $R > 0$.
		\item Given $r > 1$ the dual number of $r$ is denoted by $r' = r/(r -1)$.
	\end{itemize}
	
	\subsection{Outline} The remainder of this work is organized as follows: In the forthcoming section we consider some results concerning on some basic tools for our main problem. In Section 3 we consider some results taking into account the Nonlinear Rayleigh Quotient. The Section 4 is devoted to prove some results around the Nehari method. Here we also consider a regularity result for our main problem. Section 5 is devoted to the proof of our main result looking for the energy levels for each minimizer in the Nehari manifolds $\mathcal{N}_{\lambda}^+$ and $\mathcal{N}_{\lambda}^-$. 
	
	\section{Preliminaries}
	In this section we prove some basic properties for the energy functional $J_\lambda$ taking into account  the Nehari method and the nonlinear Raleigh quotient. Firstly, we shall prove some tools for nonlocal elliptic problems with of the Stein-Weiss type.  The main ingredient in the present work is to apply the Stein-Weiss type inequality and the weighted Hardy-Littlewood-Sobolev inequality, see \cite{Stein_weiss1958}. On this subject we refer also the reader to \cite{lieb}. These results can be stated as follows:
	\begin{prop}[Stein-Weiss type inequality]\label{Stein-Weiss}
		Let  $ N\geq 3, s_1, r_1 > 1,\mu > 0, \alpha > 0, 0<2\alpha+\mu < N$ be fixed.  Assume also that $f \in L^{r_1}(\mathbb R^N)$ and $h \in L^{s_1}(\mathbb R^N)$ where $r_1'$ and $s_1'$ are the conjugate exponent of $r_1 > 1$ and $s_1 > 1$, respectively. Then there exists a constant $C(\alpha,\mu,N,s_1,r_1) > 0$ such that
		\begin{equation*}\label{SW}
			\displaystyle\int_{\mathbb R^N}\int_
			{\mathbb R^N}\frac{f(x)h(y)}{\vert x \vert^{\alpha}\vert x-y \vert^{\mu}\vert y \vert ^{\alpha}}dxdy\leq C(\alpha,\mu,N,s_1,r_1)\Vert f \Vert_{r_1} \Vert h \Vert_{s_1},
		\end{equation*}
		where
		\begin{equation*}\label{des1}
			0 < \alpha<\min\left\{\displaystyle\frac{N}{r_1'},\displaystyle\frac{N}{s_1'}\right\}\;\;\hbox{and}\;\;\;\frac{1}{r_1}+\frac{1}{s_1}+\frac{2\alpha+\mu}{N}=2.
		\end{equation*}
	\end{prop}
	\begin{prop}[Weighted Hardy-Littlewood-Sobolev inequality] \label{DHLSP1}
		Let $s_1, t_1 >1,\, \mu > 0, \alpha > 0, \,0< 2\alpha+\mu < N$ be fixed. Assume also that $h \in L^{s_1}(\mathbb R^N)$. Then  there exists a constant $C(r_1',s_1,\alpha,N,\mu)> 0$ such that
		\begin{equation*}\label{HLSP}
			\left\Vert \displaystyle\int_{\mathbb R^N}\displaystyle\frac{h(y)}{\vert x\vert^{\alpha}\vert y-x\vert^{\mu}\vert y\vert^{\alpha}}dy\right\Vert_{r_1'}\leq C(r_1',s_1,\alpha,N,\mu)\Vert h\Vert_{s_1},
		\end{equation*}
		where 
		\begin{equation*}\label{des2}
			0 < \alpha<\min\left\{\displaystyle\frac{N}{r_1'},\displaystyle\frac{N}{s_1'}\right\}\;\;\hbox{and}\;\;\;\frac{1}{r_1}+\frac{1}{s_1}+\frac{2\alpha+\mu}{N}=2.
		\end{equation*}
	\end{prop}
	
	Now, we shall prove a powerful result in order to guarantee that Brezis-Lieb Lemma is satisfied for our setting. This kind of result is important for the proof of the functional $J_\lambda$ is in $C^1$ class. Here we follow some ideas discussed in  \cite[Lemma 2.5]{MOROZ} and \cite[Theorem 4.2.7]{WILLEM}. Hence, we consider the following result: 
	
	\begin{lemma}\label{auxiliar Brezis Lieb}
		Suppose $(H_1),(H_2),(H_3)$ where $ b\in L^{2^*m_1/[p(2^*-m_1)]}(\mathbb R^N)$ and $1<m_1<2^*$. Assume also that $u_k\rightharpoonup u$ in $H^{1}(\mathbb R^N)$ as $k\rightarrow \infty$ with $u \in H^1(\mathbb{R}^N)$. Then, for each $p\in [1,m_1]$, we obtain that  
		$$\displaystyle\lim_{n\rightarrow\infty}\displaystyle\int_{\mathbb R^N}b(x)^{\frac{m_1}{p}}\vert\vert u_k\vert^p-\vert u_k-u\vert^p-\vert u\vert^p\vert^{\frac{m_1}{p}}=0.$$
		\begin{proof}
			Let $\varepsilon>0$ be fixed. It is not hard to see that there exists $c(\varepsilon)>0$ such that 
			\begin{equation*}
				\left \vert \vert \theta_1+\theta_2\vert ^p-\vert \theta_1\vert^p-\vert \theta_2\vert^p\right\vert\leq\varepsilon\vert \theta_1\vert^p+c(\varepsilon)\vert \theta_2\vert^p
			\end{equation*}
			holds for $\theta_1,\theta_2 \in \mathbb R$.
			Consider $\theta_1=u_k-u$ and $\theta_2=u$. As a consequence, we obtain that 
			$$
			\left\vert \vert u_k\vert^p-\vert u_k-u\vert^p-\vert u\vert^p\right\vert\leq \varepsilon\vert u_k-u\vert^p+c(\varepsilon)\vert u\vert^p.
			$$
			Under these conditions, we infer that 
			\begin{eqnarray*}
				b^{\frac{m_1}{p}}\left\vert \vert u_k\vert^p-\vert u_k-u\vert^p-\vert u\vert^p\right\vert^{\frac{m_1}{p}}\nonumber & \leq & b^{\frac{m_1}{p}}\left[\varepsilon\vert u_k-u\vert^p+c(\varepsilon)\vert u\vert^p\right]^{\frac{m_1}{p}} \leq  2^{\frac{m_1}{p}} b^{\frac{m_1}{p}}\left[\varepsilon^{\frac{m_1}{p}}\vert u_k-u\vert^{m_1}+c(\varepsilon)^{\frac{m_1}{p}}\vert u\vert^{m_1}\right].
			\end{eqnarray*}
			Now, putting $c_{\varepsilon}=c^{\frac{m_1}{p}}(\varepsilon)$, we also mention that 
			$$
			2^{\frac{m_1}{p}} b^{\frac{m_1}{p}}\left[\varepsilon^{\frac{m_1}{p}}\vert u_k-u\vert^{m_1}+c_{\varepsilon}\vert u\vert^{m_1}\right]-  b^{\frac{m_1}{p}}\left\vert \vert u_k\vert^p-\vert u_k-u\vert^p-\vert u\vert^p\right\vert^{\frac{m_1}{p}}\geq 0.
			$$
			Therefore, by using the Fatou's Lemma, we deduce that
			$$
			2^{\frac{m_1}{p}}\displaystyle\int_{\mathbb R^N}\!\!\!\!b^{\frac{m_1}{p}}(x)c_{\varepsilon}\vert u\vert^{m_1}dx \leq  \liminf_{k\rightarrow \infty} \displaystyle\int_{\mathbb R^N}\!\!\!\!\left\{b^{\frac{m_1}{p}}(x)2^{\frac{m_1}{p}}\left[\varepsilon^{\frac{m_1}{p}}\vert u_k-u\vert^{m_1}+c_{\varepsilon}\vert u\vert^{m_1}\right] - b^{\frac{m_1}{p}}\left \vert \vert u_k\vert^p-\vert u_k-u\vert^p-\vert u\vert^p\right\vert^{\frac{m_1}{p}}\right\}dx.
			$$
			Hence, we obtain the following estimates
			\begin{eqnarray*}
				\nonumber 0 & \leq & \liminf_{k\rightarrow \infty} \displaystyle\int_{\mathbb R^N}\left\{b^{\frac{m_1}{p}}(x)2^{\frac{m_1}{p}}\varepsilon^{\frac{m_1}{p}}\vert u_k-u\vert^{m_1}-b^{\frac{m_1}{p}}\left \vert \vert u_k\vert^p-\vert u_k-u\vert^p-\vert u\vert^p\right\vert^{\frac{m_1}{p}}\right\}dx\\
				\nonumber & \leq & 2^{\frac{m_1}{p}}\varepsilon^{\frac{m_1}{p}}\liminf_{k\rightarrow \infty} \displaystyle\int_{\mathbb R^N}b^{\frac{m_1}{p}}(x)\vert u_k-u\vert^{m_1}-\limsup_{k\rightarrow \infty} \displaystyle\int_{\mathbb R^N} b^{\frac{m_1}{p}}\left \vert \vert u_k\vert^p-\vert u_k-u\vert^p-\vert u\vert^p\right\vert^{\frac{m_1}{p}}dx.
			\end{eqnarray*}
			Therefore, by using the H\"older inequality and Remark \ref{imersaoRN}, we infer that
			\begin{eqnarray}\label{esquema}
				\limsup_{k\rightarrow \infty} \displaystyle\int_{\mathbb R^N} b^{\frac{m_1}{p}}\left \vert \vert u_k\vert^p-\vert u_k-u\vert^p-\vert u\vert^p\right\vert^{\frac{m_1}{p}}dx  &\leq& 2^{\frac{m_1}{p}}\varepsilon^{\frac{m_1}{p}} \liminf_{k\rightarrow \infty} \displaystyle\int_{\mathbb R^N}b^{\frac{m_1}{p}}(x)\vert u_k-u\vert^{m_1} \nonumber \\
				& = & 2^{\frac{m_1}{p}}\varepsilon^{\frac{m_1}{p}}\liminf_{k\rightarrow \infty} \Vert b\Vert_{m_1 p/l}^{\frac{m_1}{p}}\Vert u_k-u\Vert^{m_1}_{2^*}\nonumber\\
				& \leq &  2^{\frac{m_1}{p}}\varepsilon^{\frac{m_1}{p}}\liminf_{k\rightarrow \infty} \Vert b\Vert^{\frac{m_1}{p}}_{\frac{m_1}{p}l} S_{2^*}^{m_1}\Vert u_k-u\Vert^{m_1}.
			\end{eqnarray}
			It is important to recall that 
			$l= 2^*/(2^*-m_1)$ and $b\in L^{m_1l/p}(\mathbb R^N)$. Here was used the fact that
			$$
			\displaystyle\frac{m_1 l}{p}=\displaystyle\frac{2^*m_1}{p(2^*-m_1)}.
			$$
			Hence, taking the limit $\varepsilon\to 0$ in \eqref{esquema}, we obtain the desired result. This ends the proof. 
		\end{proof}
	\end{lemma}

	Now, we shall prove identity of Brezis-Lieb for our main problem taking into account the potentials $a,b : \mathbb{R}^N \to \mathbb{R}$. Here we borrow some ideas discussed in  \cite[Lemma 2.5]{MOROZ} and \cite[Theorem 4.2.7]{WILLEM}. Namely, we prove the following result:
	
	\begin{prop}[The Brezis-Lieb Lemma]\label{Brezis-Lieb}
		Suppose $(H_1),(H_2), (H_3)$. Assume also that $u_k\rightharpoonup u$ in $H^{1}(\mathbb R^N)$ where $u \in H^1(\mathbb{R}^N)$. Then we obtain the following assertion:
		$$
		\displaystyle\int_{\mathbb R^N}\int_{\mathbb R^N}\frac{b(x)\vert u_k(x) \vert^p b(y)\vert u_k(y)\vert^p}{\vert x\vert^\alpha\vert x-y\vert^{\mu}\vert y\vert^{\alpha}}dxdy-\displaystyle\int_{\mathbb R^N}\int_{\mathbb R^N}\frac{b(x)\vert (u_k-u)(x) \vert^p b(y)\vert (u_k-u)(y)\vert^p}{\vert x\vert^\alpha\vert x-y\vert^{\mu}\vert y\vert^{\alpha}}dxdy
		$$
		$$
		\rightarrow\displaystyle\int_{\mathbb R^N}\int_{\mathbb R^N}\frac{b(x)\vert u(x) \vert^pb(y)\vert u(y)\vert^p}{\vert x\vert^\alpha\vert x-y\vert^{\mu}\vert y\vert^{\alpha}}dxdy.
		$$
	\end{prop}
	\begin{proof}
		It is not hard to see that 
		\begin{eqnarray*}
			\nonumber \displaystyle\int_{\mathbb R^N}\int_{\mathbb R^N}\frac{b(x)\vert u_k(x) \vert^pb(y)\vert u_k(y)\vert^p}{\vert x\vert^\alpha\vert x-y\vert^{\mu}\vert y\vert^{\alpha}}dxdy &-& \displaystyle\int_{\mathbb R^N}\int_{\mathbb R^N}\frac{b(x)\vert (u_k-u)(x) \vert^pb(y)\vert (u_k-u)(y)\vert^p}{\vert x\vert^\alpha\vert x-y\vert^{\mu}\vert y\vert^{\alpha}}dxdy\\
			& = &\displaystyle\int_{\mathbb R^N}\left[\int_{\mathbb R^N}\frac{b(y)\left(\vert u_k\vert^p-\vert u_k-u\vert^p\right)}{\vert x\vert^\alpha\vert x-y\vert^{\mu}\vert y\vert^{\alpha}}dy\right] b(x)\left[\vert u_k\vert^p-\vert u_k-u\vert^p\right]dx\\
			&+& 2\displaystyle\int_{\mathbb R^N}\left[\int_{\mathbb R^N}\frac{b(y)\left(\vert u_k\vert^p-\vert u_k-u\vert^p\right)}{\vert x\vert^\alpha\vert x-y\vert^{\mu}\vert y\vert^{\alpha}}dy\right] b(x)\vert u_k-u\vert^p dx. 
		\end{eqnarray*}
		Now, we consider 
		\begin{equation}\label{I}
			I=\displaystyle\int_{\mathbb R^N}\left[\int_{\mathbb R^N}\frac{b(y)\left(\vert u_k\vert^p-\vert u_k-u\vert^p\right)}{\vert x\vert^\alpha\vert x-y\vert^{\mu}\vert y\vert^{\alpha}}dy\right] b(x)\left[\vert u_k\vert^p-\vert u_k-u\vert^p\right]dx.
		\end{equation}
		\begin{equation}\label{IIai}
			II=\displaystyle\int_{\mathbb R^N}\left[\int_{\mathbb R^N}\frac{b(y)\left(\vert u_k\vert^p-\vert u_k-u\vert^p\right)}{\vert x\vert^\alpha\vert x-y\vert^{\mu}\vert y\vert^{\alpha}}dy\right] b(x)\vert u_k-u\vert^p dx.
		\end{equation}
		Hence, we rewrite \eqref{I} in the following form
		\begin{eqnarray}\label{des.Isaco}
			I \nonumber& = & \displaystyle\int_{\mathbb R^N}\left[\int_{\mathbb R^N}\frac{b(y)\left(\vert u_k\vert^p-\vert u_k-u\vert^p\right)}{\vert x\vert^\alpha\vert x-y\vert^{\mu}\vert y\vert^{\alpha}}dy\right] b(x)\left[\vert u_k\vert^p-\vert u_k-u\vert^p\right]dx\\
			\nonumber& = &  \underbrace{\displaystyle\int_{\mathbb R^N} \left[\displaystyle\int_{\mathbb R^N}\displaystyle\frac{b(y)(\vert u_k\vert^p-\vert u_k-u\vert^p-\vert u\vert^p)}{\vert x\vert^{\alpha}\vert x-y \vert^{\mu}\vert y\vert^{\alpha}}dy\right]b(x)\left[\vert u_k\vert^p-\vert u_k-u\vert^p-\vert u\vert^p\right]dx}_{A}\\
			\nonumber &+& \underbrace{\displaystyle\int_{\mathbb R^N} \left[\displaystyle\int_{\mathbb R^N}\displaystyle\frac{b(y)(\vert u_k\vert^p-\vert u_k-u\vert^p-\vert u\vert^p)}{\vert x\vert^{\alpha}\vert x-y \vert^{\mu}\vert y\vert^{\alpha}}dy\right]b(x)\vert u\vert^p dx}_{B}\\
			\nonumber &+& \underbrace{\displaystyle\int_{\mathbb R^N} \left[\displaystyle\int_{\mathbb R^N}\displaystyle\frac{b(y)\vert u\vert^p}{\vert x\vert^{\alpha}\vert x-y \vert^{\mu}\vert y\vert^{\alpha}}dy\right]b(x)\left[\vert u_k\vert^p-\vert u_k-u\vert^p-\vert u\vert^p\right]dx}_{C}\\
			&+& \displaystyle\int_{\mathbb R^N} \left[\displaystyle\int_{\mathbb R^N}\displaystyle\frac{b(y)\vert u\vert^p}{\vert x\vert^{\alpha}\vert x-y \vert^{\mu}\vert y\vert^{\alpha}}dy\right]b(x)\vert u\vert^p dx. 
		\end{eqnarray}
		From now on, we shall analyze the terms $A, B$ and $C$ which are given by \eqref{des.Isaco}. Firstly, by using H\"older inequality together with Proposition \ref{DHLSP1} and Lemma \ref{auxiliar Brezis Lieb}, we obtain that
		\begin{eqnarray*}\label{A}
			A\nonumber &=&  \displaystyle\int_{\mathbb R^N} \left[\displaystyle\int_{\mathbb R^N}\displaystyle\frac{b(\vert u_k\vert^p-\vert u_k-u\vert^p-\vert u\vert^p)}{\vert x\vert^{\alpha}\vert x-y \vert^{\mu}\vert y\vert^{\alpha}}dy\right]b(x)\left[\vert u_k\vert^p-\vert u_k-u\vert^p-\vert u\vert^p\right]dx\\
			& \leq & \left\Vert \int_{\mathbb R^N}\frac{b(\vert u_k\vert^p-\vert u_k-u\vert^p-\vert u\vert^p)}{\vert x\vert^{\alpha}\vert x-y \vert^{\mu}\vert y\vert^{\alpha}}dy\right\Vert_{s'_1}\Vert b \left(\vert u_k\vert^p-\vert u_k-u\vert^p-\vert u\vert^p\right)\Vert_{s_1}\nonumber\\
			& \leq & C_1\left\Vert b \left(\vert u_k\vert^p-\vert u_k-u\vert^p-\vert u\vert^p\right)\Vert_{s_1}\right\Vert b \left(\vert u_k\vert^p-\vert u_k-u\vert^p-\vert u\vert^p\right)\Vert_{s_1}\displaystyle\rightarrow 0.
		\end{eqnarray*}
		where $s_1 = m_1/p$, $s_1' = m_1/(m_1 - p)$. In particular, we consider $m_1 = 2 N p/(2 N - 2 \alpha - \mu)$.
		
		Now, for the term $B$ given in \eqref{des.Isaco}. In order to do that we apply the Proposition \ref{DHLSP1} and H\"older inequality which imply that
		\begin{eqnarray}\label{ali}
			B& = &\displaystyle\int_{\mathbb R^N} \left[\displaystyle\int_{\mathbb R^N}\displaystyle\frac{b(y)(\vert u_k\vert^p-\vert u_k-u\vert^p-\vert u\vert^p)}{\vert x\vert^{\alpha}\vert x-y \vert^{\mu}\vert y\vert^{\alpha}}dy\right]b(x)\vert u\vert^p dx \nonumber \\
			& \leq & C_1\left \Vert b [\vert u_k\vert^p-\vert u_k-u\vert^p-\vert u\vert^p] \vert\right\Vert_{s_1}\left(\displaystyle\int_{\mathbb R^N}b^{l_1}\vert u(x)\vert^{pl_1}dx\right)^{\frac{1}{l_1}}
		\end{eqnarray} 
		where
		$
		s_1= 2N/(2N-2\alpha-\mu).
		$
		Once again, we obtain that
		$l_1= s_1$ and $m_1 = 2 Np/(2N - 2 \alpha - \mu)$.
		Now, by using H\"older inequality and the Remark \ref{imersaoRN}, we deduce that
		\begin{eqnarray}\label{B1}
			\nonumber&&\displaystyle\int_{\mathbb R^N}b^{l_1}\vert u\vert^{pl_1} dx \leq\left(\displaystyle\int_{\mathbb R^N}b^{l_1l_2} dx\right)^\frac{1}{l_2}\left(\displaystyle\int_{\mathbb R^N}\vert u\vert^{pl_1\frac{2^*-\epsilon}{pl_1}} dx\right)^\frac{pl_1}{2^*-\epsilon}
			= \Vert b\Vert^{l_1}_{l_1l_2} \|u \|_{2^*-\epsilon}^{pl_1}\leq S_{2^*-\epsilon}^{pl_1}\Vert b\Vert^{l_1}_{l_1l_2}\Vert u\Vert^{pl_1}<\infty
		\end{eqnarray}
		where $\epsilon > 0$ is small enough. Here was used the fact that $b\in L^{l_1l_2}(\mathbb R^N)$.  In fact, we mention that $2^*/pl_1>1$ and  
		$$l_2=\displaystyle\frac{[2N - \epsilon(N-2)][2N-2\alpha-\mu]}{2N[2N-2\alpha-\mu-p(N-2)] - \epsilon(N-2)(2N - 2 \alpha - \mu)}.$$
		As a consequence, we obtain that 
		$$l_1l_2=\displaystyle\frac{2N[2N - \epsilon(N-2)]}{2N[2N-2\alpha-\mu-p(N-2)]- \epsilon(N-2)(2N - 2 \alpha - \mu)} = \sigma + \eta$$
		where $\sigma = 2N/(2N - 2 \alpha - \mu - p (N -2))$ and $\eta > 0$ is small enough.
		Therefore, by using \eqref{ali} and taking into account the last estimates, we obtain that 
		\begin{equation}\label{final de b}
			B \leq C \Vert b [\vert u_k\vert^p-\vert u_k-u\vert^p-\vert u\vert^p]\Vert_{s_1}\Vert b\Vert^{l_1}_{l_1l_2}\Vert u \Vert^{pl_1}\rightarrow 0.
		\end{equation}
		Furthermore, by applying Proposition \ref{Stein-Weiss}, we infer that 
		\begin{eqnarray}\label{C}  
			C\nonumber &=&\displaystyle\int_{\mathbb R^N} \left[\displaystyle\int_{\mathbb R^N}\displaystyle\frac{b(y)\vert u\vert^p}{\vert x\vert^{\alpha}\vert x-y \vert^{\mu}\vert y\vert^{\alpha}}dy\right]b(x)\left[\vert u_k\vert^p-\vert u_k-u\vert^p-\vert u\vert^p\right]dx\\
			&\leq & C_2\left\Vert b \vert u \vert^p\right\Vert_{l_1}\left\Vert b [ \vert u_k\vert^p-\vert u_k-u\vert^p-\vert u\vert^p]\vert\right\Vert_{s_1}.
		\end{eqnarray}
		Once again by using the H\"older inequality and Proposition \ref{imersaoRN}, we mention that
		\begin{eqnarray}\label{b-u}
			\Vert b\vert u\vert^p\Vert^{l_1}_{l_1}&\leq& \left(\displaystyle\int_{\mathbb R^N}\vert b\vert^{l_1l_2}dx\right)^{\frac{1}{l_2}}\left(\displaystyle\int_{\mathbb R^N}\vert u\vert^{pl_1\frac{2^*- \epsilon}{pl_1}}dx\right)^{\frac{pl_1}{2^*- \epsilon}}
			=\Vert b\Vert^{l_1}_{l_1l_2}\Vert u\Vert^{pl_1}_{2^*-\epsilon}\leq \Vert b\Vert^{l_1}_{l_1l_2}\Vert u \Vert^{pl_1}<\infty.
		\end{eqnarray}
		Here we emphasize that $ b\in L^{l_1l_2}(\mathbb R^N)$ where
		$l_2=\left((2^* - \epsilon)/pl_1\right)'$ and $2^*/pl_1>1$.   
		Therefore, by using \eqref{C}, we infer that 
		\begin{eqnarray}\label{final de C}
			C\leq C_2\left\Vert b\vert u \vert^p\right\Vert_{l_1}\left\Vert b [\vert u_k\vert^p-\vert u_k-u\vert^p-\vert u\vert^p]\right\Vert_{s_1}\rightarrow 0.    
		\end{eqnarray}
		Under these conditions, by using \eqref{final de b}, \eqref{final de C} and \eqref{des.Isaco}, we obtain that 
		$$I\to \displaystyle\int_{\mathbb R^N} \left[\displaystyle\int_{\mathbb R^N}\displaystyle\frac{b(y)\vert u\vert^p}{\vert x\vert^{\alpha}\vert x-y \vert^{\mu}\vert y\vert^{\alpha}}dy\right]b(x)\vert u\vert^p dx.$$
		
		At this stage, we shall analyze the term II given in \eqref{IIai}. According to Proposition \ref{Stein-Weiss} we deduce that 
		\begin{eqnarray}\label{II}
			II \nonumber& = &\displaystyle\int_{\mathbb R^N}\left[\int_{\mathbb R^N}\frac{b(y)\left(\vert u_k\vert^p-\vert u_k-u\vert^p\right)}{\vert x\vert^\alpha\vert x-y\vert^{\mu}\vert y\vert^{\alpha}}dy\right] b(x)\vert u_k-u\vert^p dx\\
			\nonumber & \leq &  C\left\Vert b\left(\vert u_k\vert^p-\vert u_k-u\vert^p\right)\right\Vert_{s_1}\left\Vert b\vert u_k-u\vert^p\right\Vert_{l_1}\\
			\nonumber & = & C\left\Vert b\left(\vert u_k\vert^p-\vert u_k-u\vert^p-\vert u\vert^p+\vert u\vert^p\right)\right\Vert_{s_1}\left\Vert b\vert u_k-u\vert^p\right\Vert_{l_1}\\
			& \leq & C\left\Vert b\left(\vert u_k\vert^p-\vert u_k-u\vert^p-\vert u\vert^p\right)\right\Vert_{s_1}\left\Vert b\vert u_k-u\vert^p\right\Vert_{l_1} + C\left\Vert b\vert u\vert^p\right\Vert_{s_1}\left\Vert b\vert u_k-u\vert^p\right\Vert_{l_1}.
		\end{eqnarray}
		Recall also that $u_k\rightharpoonup u$ in $H^1(\mathbb R ^N)$. Then there exists $ h_l\in L^{l}(B(0,R)),\;\; 1\leq l< 2^*$ such that
		$u_k \rightarrow u$ in $L^{l}(B(0,R))$,
		$u_k(x) \rightarrow u(x)$ a. e. in  $B(0,R)$ and 
		$|u_k| \le h_l$ in $\mathbb{R}^N$.
		In light of H\"older inequality and Proposition \ref{imersaomega}, we deduce also that 
		\begin{eqnarray}\label{u_k-u}
			\nonumber \Vert b \vert u_k-u\vert^p\Vert_{l_1}^{l_1}  &=&  \displaystyle \int_{B(0,R)}b^{l_1}(x)\vert u_k-u\vert^{pl_1}dx+\displaystyle\int_{\mathbb R^N\backslash B(0,R)}b^{l_1}(x)\vert u_k-u\vert^{pl_1}dx\\
			\nonumber& \leq & \Vert b\Vert^{l_1}_{l_1l_2}\left(\displaystyle\int_{B(0,R)}\vert u_k-u\vert^{pl_1\frac{2^*-\epsilon}{pl_1}}dx\right)^{\frac{pl_1}{2^*-\epsilon}}+c\Vert u_k-u\Vert^{pl_1}\left(\displaystyle\int_{\mathbb R^N\backslash B(0,R)}b^{l_1l_2}(x)dx\right)^{\frac{1}{l_2}} \nonumber \\
			&\leq& \Vert b\Vert^{l_1}_{l_1l_2}\left(\displaystyle\int_{B(0,R)}\vert u_k-u\vert^{2^*-\epsilon}dx\right)^{\frac{pl_1}{2^*-\epsilon}}+ C \left(\displaystyle\int_{\mathbb R^N\backslash B(0,R)}b^{l_1l_2}(x)dx\right)^{\frac{1}{l_2}}. 
		\end{eqnarray}
		Notice also that $b\in L^{l_1l_2}(\mathbb R^N)$. In particular,  we obtain that
		$\left\vert b^{l_1l_2}(x)\chi_{\mathbb R^N \backslash B(0,R)}(x)\right\vert\leq b^{l_1l_2}\in L^1(\mathbb R^N).$
		The last estimate allows us to apply the Dominated Convergence Theorem proving that 
		\begin{eqnarray}\label{B11}
			\int_{\mathbb R^N\backslash B(0,R)}b^{l_1l_2}(x)dx \rightarrow 0, \,\, \mbox{as} \,\,  R \to \infty.
		\end{eqnarray}
		On the other hand, for each $\epsilon > 0$ small enough and taking into account Proposition \ref{imersaomega}, we obtain also that 
		\begin{equation}\label{bai}
			\displaystyle\int_{B(0,R)}\vert u_k-u\vert^{2^*-\epsilon}dx \to 0, \,\, k \to \infty.
		\end{equation}
		Therefore,  by using \eqref{u_k-u}, \eqref{B11} and \eqref{bai}, we obtain that $\Vert b \vert u_k-u\vert^p\Vert_{l_1}^{l_1} \to 0$ as $k \to \infty$.
		In order to estimate II,  we apply \eqref{b-u} and $s_1=l_1$ in order to conclude that
		$\Vert b\vert u\vert^p\Vert^{s_1}_{s_1} < \infty$. 
		In view of \eqref{II} and the estimates just above we conclude that 
		$II \rightarrow 0$ as $k \to \infty$. As a consequence, we infer that  
		\begin{eqnarray}
			\nonumber&& \displaystyle\int_{\mathbb R^N}\left[\int_{\mathbb R^N}\frac{b(y)\left(\vert u_k\vert^p-\vert u_k-u\vert^p\right)}{\vert x\vert^\alpha\vert x-y\vert^{\mu}\vert y\vert^{\alpha}}dy\right] b(x)\left[\vert u_k\vert^p-\vert u_k-u\vert^p\right]dx \rightarrow  \displaystyle\int_{\mathbb R^N} \left[\displaystyle\int_{\mathbb R^N}\displaystyle\frac{b(y)\vert u\vert^p}{\vert x\vert^{\alpha}\vert x-y \vert^{\mu}\vert y\vert^{\alpha}}\right]dyb(x)\vert u\vert^p dx. 
		\end{eqnarray}
		This ends the proof. 
	\end{proof}
	
	\section{The nonlinear Rayleigh quotient}
	In this section we shall consider the nonlinear Rayleigh quotient for our main problem. The main idea here is to find the largest positive number $\lambda^* > 0$ such that the Nehari method can be applied. Consider the functionals $R_n,~R_e : H^1(\mathbb R^N)\backslash \{0\}\rightarrow \mathbb R$ associated with the parameter $\lambda > 0$ which are defined in \eqref{R_n}. It is important to stress that \
	$tu \in\mathcal N_{\lambda}$ if and only if $R_n(tu)=\lambda$ where $u \in H^1(\mathbb R^N)\backslash \{0\}$ and $t > 0$. Similarly, we observe that
	$J_{\lambda}(tu)=0$ if and only if $R_e(u)=\lambda$ where $u \in H^1(\mathbb R^N)\backslash \{0\}$ and $t > 0$.  
	Notice also that the functionals $R_e, R_n$ belong to 
	$C^1(H^1(\mathbb R^N)\backslash \{0\};\mathbb R)$. Moreover, the functionals $R_e, R_n$ are related with the energy functional $J_{\lambda}$ and its derivatives. More precisely, we consider the following results:
	\begin{prop}\label{relação entre R_n e J'}
		Let  $u\in H^1(\mathbb R^N)\backslash \{0\}$ be fixed. Then using \eqref{R_n} we obtain that $R_n(u)=\lambda$ if, and only if, $J'_{\lambda}(u)u=0$. Moreover, we obtain that $R_n(u)>\lambda$ if, and only if, $J_{\lambda}'(u)u>0$. It holds also that 
		$R_n(u)<\lambda$ if, and only if, $J'_{\lambda}(u)u<0$.
	\end{prop}
	\begin{prop}\label{relação entre R_e e J}
		Let  $u\in H^1(\mathbb R^N)\backslash \{0\}$ be fixed. Then using \eqref{R_n} we obtain that
		$R_e(u)=\lambda$ if, and only if, $J_{\lambda}(u)=0$. Moreover, we obtain that $R_e(u)>\lambda$ if, and only if, $J_{\lambda}(u)>0$. It holds also that $R_e(u)<\lambda$ if, and only if, $J_{\lambda}(u)<0$.
	\end{prop}
	Under these conditions, we need to consider the behavior for the functional $R_n$ and $R_e$ which allows us to ensure existence of ground state solutions and bound state solutions our main problem. Notice that is crucial in our arguments to find unique projections in the Nehari sets $\mathcal{N}_{\lambda}^-$ and $\mathcal{N}_{\lambda}^+$ for each $u\in H^1(\mathbb R^N)\backslash \{0\}$ with $\lambda \in (0,\lambda^*)$. Moreover, the functional $R_e$ given in \eqref{R_n} is used in order to decide the sign for solutions which lies in $\mathcal{N}_{\lambda}^-$.
	
	Now, we shall consider the fibering function for the functional $R_n$ defining the function $Q_n : [0,\infty) \to \mathbb{R}$ given by
	\begin{equation*}\label{Q_n}
		Q_n(t)=R_n(tu)=\frac{t^{2-q}\Vert u\Vert^2-t^{2p-q}B(u)}{A(u)},~t \geq 0.
	\end{equation*}
	It follows from the last identity that 
	\begin{equation}\label{Q'_n}
		Q'_n(t)=\frac{(2-q)t^{1-q}\Vert u\Vert^2-(2p-q)t^{2p-q-1}B(u)}{A(u)},~t>0.
	\end{equation}
	It is easy to see that the critical point of $Q_n$ provide us a functional $t_n:H^1(\mathbb R^N)\backslash\{0\}\to \mathbb{R}$ given by 
	\begin{equation*}\label{t_n}
		t_n(u)=\left[\frac{(2-q)\Vert u\Vert^2}{(2p-q)B(u)}\right]^{\frac{1}{2p-2}}.
	\end{equation*}
	Hence, we can define the functional $\Lambda_n : H^1(\mathbb R^N)\backslash\{0\}\to \mathbb{R}$ given by 
	\begin{equation*}\label{Q_n(t_n)}
		\Lambda_n (u) = Q_n(t_n(u))=\frac{C_{p,q}\Vert u\Vert^{\frac{2p-q}{p-1}}}{A(u)\left[B(u)\right]^{\frac{2-q}{2p-2}}},
	\end{equation*}
	where
	\begin{equation*}
		C_{p,q}=\left(\frac{2-q}{2p-q}\right)^{\frac{2-q}{2p-2}}\left(\frac{2p-2}{2p-q}\right).
	\end{equation*}
	It important to analyze the behavior of $Q_n$ and $Q'_n$ which allows to show that $Q_n$ admits an unique critical point. Namely, we mention that
	$$
	\displaystyle\lim_{t\to 0}\frac{Q_n(t)}{t^{2-q}}= \displaystyle\frac{\Vert u \Vert^2}{A(u)}>0,\, \,
	\displaystyle\lim_{t\rightarrow \infty}\frac{Q_n(t)}{t^{2p-q}}=
	-\frac{B(u)}{A(u)}<0,
	$$
	and
	$$
	\displaystyle\lim_{t\rightarrow 0}\frac{Q'_n(t)}{t^{1-q}}= \displaystyle\frac{(2-q)\Vert u \Vert^2}{A(u)}>0,\,\,
	\displaystyle\lim_{t\to\infty}\frac{Q'_n(t)}{t^{2p-q-1}}= \displaystyle -\frac{(2-q)B(u)}{A(u)}<0.
	$$
	In particular, we obtain that $Q'_n(t)>0$ holds for each $t\in (0,t_n(u))$ and $Q'_n(t)<0$ holds for each $t>t_n(u)$. Recall also that $Q_n(0)=0$ and $Q_n'(t_n(u)) = 0$. Thus, the number $t_n(u)>0$ is the unique critical point for $Q_n$ and $ Q_n(t_n(u))=\displaystyle\max_{t>0}Q_n(t)$.
	It is important to point out that the function $Q_e : [0,\infty) \to \mathbb{R}$ has a similar behavior. In fact, we consider the function
	\begin{equation*}\label{Q_e}
		Q_e(t)=R_e(tu)=\displaystyle q\frac{\displaystyle\frac{1}{2}\Vert u\Vert^2t^{2-q}-\displaystyle\frac{t^{2p-q}}{2p}B(u)}{A(u)},~t \geq 0.
	\end{equation*}
	Analogously, the first derivative of $Q_e$ is given by
	\begin{equation*}\label{Q'_e}
		Q'_e(t)=\displaystyle\frac{\displaystyle\frac{(2-q)}{2}t^{1-q}\Vert u\Vert^2-\displaystyle\frac{(2p-q)}{2p}t^{2p-q-1}B(u)}{\displaystyle\frac{1}{q}A(u)},~t>0.
	\end{equation*}
	It is easy to see that the critical point of $Q_e$ ensures the existence of a functional $t_e:H^1(\mathbb R^N)\backslash\{0\}\rightarrow \mathbb R$ given in the following form:
	\begin{equation*}\label{t_e}
		t_e(u)=\left[\displaystyle\frac{p(2-q)\Vert u\Vert^2}{(2p-q)B(u)}\right]^{\frac{1}{2p-2}}.
	\end{equation*}
	Hence, we consider the functional $\Lambda_e : H^1(\mathbb R^N)\backslash\{0\}\rightarrow \mathbb{R}$ given by
	\begin{equation*}\label{Q_e(t_e)}
		\Lambda_e(u) = Q_e(t_e(u))=\displaystyle\frac{\widetilde{C}_{p,q}\Vert u\Vert^{\frac{2p-q}{p-1}}}{ A(u) B(u)^{\frac{2-q}{2p-2}}},
	\end{equation*}
	where
	$$
	\widetilde{C}_{p,q}=p^{\displaystyle\frac{2-q}{2p-2}}\left(\displaystyle\frac{2-q}{2p-q}\right)^{\displaystyle\frac{2-q}{2p-2}}q\left(\displaystyle\frac{p-1}{2p-q}\right)
	$$
	
	It is is worthwhile to mention that we need to guarantee existence of an optimal $\lambda^* > 0$ in such way that $\mathcal N_{\lambda}^0=\emptyset$ is satisfied for each $\lambda \in (0, \lambda^*)$. Under these conditions, we shall prove that any nonzero function has projection in $\mathcal{N}_{\lambda}^+$ and $\mathcal{N}_{\lambda}^-$. To do this we need to consider the functionals $\Lambda_n$ and $\Lambda_e$ which are in $C^1$ class. As a consequence, we can prove the following result:
	
	\begin{prop}\label{relação ente lambda n e lambda e}
		Suppose ($H_1$)$-$($H_3$). Then $\Lambda_e(u)=C\Lambda_n(u)$ holds for some $C \in (0,1)$.
		\begin{proof}
			Firstly, we observe that $$\displaystyle\frac{\Lambda_e(u)}{\Lambda_n(u)}=\displaystyle\frac{\widetilde{C}_{p,q}}{C_{p,q}}=C $$ where $C > 0$. It is not hard to verify that
			$$\widetilde{C}_{p,q}=\displaystyle\frac{q p^{\frac{2-q}{2(p-1)}}C_{p,q}}{2}.$$ 
			Hence, we need to prove that $q p^{\frac{2-q}{2(p-1)}}<2$ holds. The last inequality is equivalent to $$\displaystyle\frac{(2-q)}{2(p-1)}\ln p+\ln q-ln 2<0.$$ Define the function $f : [0, \infty) \to \mathbb{R}$ in the following form: $$f(q)=\displaystyle\frac{(2-q)}{2(p-1)}\ln p+\ln q-ln 2.$$
			Recall that $f(q) < 0$ if only and only if $q p^{\frac{2-q}{2(p-1)}}<2$ holds where $1 \leq q < 2$. Notice that $f$ is in $C^1$ class. Furthermore, we observe that $f(1)=\ln p/[2(p-1)]-\ln 2$ and $f(2)=0$.
			Here we observe that $$f'(q)=-\displaystyle\frac{ln  p}{2(p-1)}+\displaystyle\frac{1}{q}\;\;\hbox{and}\;\; f''(q)=-\displaystyle\frac{1}{q^2}<0.$$ Hence $f$ is a concave function. Moreover, using the fact that $\ln p<p-1$, we obtain that  $f(1)<0$. Notice also that $f'(q)>0$ holds for each $1\leq q<2$. In fact, for each $1\leq q<2$, we know that 
			$$
			f'(q)=-\displaystyle\frac{ln p}{2(p-1)}+\displaystyle\frac{1}{q}>-\displaystyle\frac{p-1}{2(p-1)}+\displaystyle\frac{1}{q}=-\displaystyle\frac{1}{2}+\displaystyle\frac{1}{q}>0.
			$$
			Thus,  the function $t \mapsto f(t)$ is increasing for $1\leq q<2$. The last assertion implies that $f(q)< f(2) = 0$ holds for each $1\leq q<2$. Therefore, we obtain that $C<1$ showing the desired result. 
		\end{proof}
	\end{prop}
	At this stage, we shall prove a relation between the derivative of the fiber function $t \mapsto Q'_n(t)$ and the second derivative of the energy function $J_\lambda$. In order to do that we consider the functional $A:H^1(\mathbb R^N)\to \mathbb R$ which in given by \eqref{A(u)}. Recall that $A$ is in $C^1$ class and the Gauteaux derivative of $A$ is given by  
	\begin{eqnarray*}
		A'(u)\varphi~=~ H(u,\varphi) =  q\displaystyle\int_{\mathbb R^N}a(x)\vert u\vert^{q-2}u\varphi dx,  u\in H^1(\R^N)\backslash \{0\}, \varphi \in H^1(\R^N).
	\end{eqnarray*}
	Similarly, the functional $B : H^1(\mathbb R^N)\to \mathbb R$ which in given by \eqref{B(u)} is also in $C^1$ class. For this functional the Gauteaux derivative is given by 
	\begin{equation*}\label{D(u)}
		B'(u) \varphi~ = ~2p D(u,\varphi)=\displaystyle\int\limits_{\mathbb R^N}\int\limits_{\mathbb R^N}\frac{b(y)\vert u(y)\vert^p}{\vert x\vert^\alpha\vert x-y\vert^\mu \vert y\vert^\alpha}b(x)\vert u(x)\vert^{p-2}u\varphi\;dx dy,\;\;u,\varphi\in H^1(\mathbb R^N).
	\end{equation*}
	Let $u\in H^1(\R^N)\backslash \{0\}$ be a fixed function such that $R_n(tu)=\lambda$ with $t > 0$. Using \eqref{Q'_n}, \eqref{A(u)}, \eqref{B(u)} we have that
	\begin{eqnarray}\label{C1}
		\nonumber tA'(tu)tu\frac{d}{dt}R_n(tu) =  (2-q)t^2\Vert u\Vert^2-(2p-q)t^{2p}B(u).
	\end{eqnarray}
	On the other hand, by using \eqref{J''}, it follows that
	\begin{eqnarray}\label{D2}
		J_{\lambda}''(tu)(tu,tu)  & = & t^2\Vert u\Vert^2-\lambda(q-1)t^q A(u)
		- (2p-1)t^{2p} B(u).
	\end{eqnarray}
	Now, by using the fact that $\lambda=R_n(tu)$ and \eqref{D2}, we infer that
	$
	J_{\lambda}''(tu)(tu,tu) \nonumber = (2-q)t^2\Vert u\Vert^2-(2p-q)t^{2p}B(u).
	$
	Therefore, by using \eqref{C1}, for each $u\in H^1(\R^N)\backslash \{0\}$ such that $R_n(tu)=\lambda, t > 0$, we obtain that
	\begin{equation*}\label{R_n e J''}
		\displaystyle\frac{d}{dt}R_n(tu)=\displaystyle\frac{q}{t}\frac{J_{\lambda}''(tu)(tu,tu)}{A'(tu)tu}.
	\end{equation*}
	Under these conditions, by using the last identity, we are able to prove the following result:
	\begin{prop}\label{d/dtR_n(tu) e J''}
		Suppose ($H_1$)$-$($H_3$). Let $u \in H^1(\mathbb R^N)\backslash \{0\}$ be a fixed function such that $R_n(tu)=\lambda$ holds for some $t>0$. Then we obtain that $\frac{d}{dt}R_n(tu)>0$ if and only if $ J_{\lambda}''(tu)(tu,tu)>0$. Furthermore, we observe that $\frac{d}{dt}R_n(tu)<0$ if and only if $J_{\lambda}''(tu)(tu,tu)<0$. It also holds that $\frac{d}{dt}R_n(tu)=0$ if and only if $ J_{\lambda}''(tu)(tu,tu)=0$.
	\end{prop}
	Similarly, we shall prove a relationship between the first derivative of function $Q_e$ and the first derivative of the energy functional $J_\lambda$. Namely, for each $u \in H^1(\mathbb R^N)\backslash \{0\}$ such that $R_e(tu)=\lambda$ and taking into account \eqref{J'}, we deduce that  
	\begin{equation}\label{D4}
		\frac{d}{dt}R_e(tu)=\frac{q}{t}\frac{J_{\lambda}'(tu)tu}{A(tu)}, t > 0.
	\end{equation}
	As a consequence, by using \eqref{D4}, we can prove the following result:
	\begin{prop}\label{relação entre R_n e J''}
		Suppose ($H_1$)$-$($H_3$). Let $u \in H^1(\mathbb R^N)\backslash \{0\}$ be a fixed function such that $R_e(tu)=\lambda$ for some $t>0$. Then we obtain that $\frac{d}{dt}R_e(tu)>0$ if and only if $ J_{\lambda}'(tu)tu>0$. Furthermore, we obtain that $\frac{d}{dt}R_e(tu)<0$ if and only if $J_{\lambda}'(tu)tu<0$. It holds also that $\frac{d}{dt}R_e(tu)=0$ if and only if $J_{\lambda}'(tu)tu=0$.
	\end{prop}
	\begin{rmk} \label{obs} Let $u \in H^1(\mathbb R^N)\backslash \{0\}$ be a fixed function. Hence, by using some straightforward identities, we deduce that $Q_n(t)-Q_e(t)=\frac{t}{q}\frac{d}{dt}Q_e(t)$ holds for each $t>0$.
	\end{rmk}
	As a consequence, by using Remark \ref{obs}, we can prove the following result:
	\begin{prop}\label{relação entre Q_n e Q_e}
		Suppose ($H_1$)$-$($H_3$). Then $\frac{d}{dt}Q_e(t)>0$ if and only if $Q_n(t)>Q_e(t)$ which makes sense only for $t\in (0,t_e(u))$. Furthermore, $\frac{d}{dt}Q_e(t)<0$ if and only if $Q_n(t)<Q_e(t)$ which is satisfied only for $t\in (t_e(u),\infty)$. Moreover, 
		$\frac{d}{dt}Q_e(t)=0$ if and only if $Q_n(t)=Q_e(t)$ which is verified only for $t=t_e(u)$.
	\end{prop}
	Now, we shall consider a Brezis-Lieb result for the functional $A : H^1(\mathbb{R}^N) \to \mathbb{R}$. Namely, we can prove the following result:
	\begin{prop}\label{A(x)}
		Suppose ($H_1$)$-$($H_3$). Let $(u_k)$ be is a bounded sequence in $L^m(\mathbb R^N)$ in such way that $u_k\rightharpoonup u$ in $H^1(\mathbb R^N)$. Then we obtain that 
		$$
		\displaystyle\int_{\mathbb R^N}a(x)\vert u_k\vert^q dx -\displaystyle\int_{\mathbb R^N}a(x)\vert u_k-u\vert^q dx\rightarrow \displaystyle\int_{\mathbb R^N}a(x)\vert u\vert^q dx 
		$$
		\begin{proof}
			The proof follows using the same ideas discussed in \cite{LIEB1983}.
			We omit the details.
		\end{proof}
	\end{prop}
	
	In the sequel, we shall consider some results taking into account the potentials $a, b : \mathbb{R}^N \to \mathbb{R}$. Firstly, we prove the following result: 
	\begin{prop}\label{c1} 	Suppose ($H_1$)$-$($H_3$). Let $(v_k)\in H^1(\mathbb R ^N)$ be a bounded sequence such that $v_k\rightharpoonup v$ in $H^1(\mathbb R ^N$). Then we obtain that 
		$$
		\displaystyle\lim_{k\longrightarrow\infty}\int_{\mathbb{R}^N}a(x)\vert v_k\vert^qdx=\int_{\mathbb{R}^N}a(x)\vert v\vert ^qdx
		$$
		\begin{proof} Firstly, we observe that 
			\begin{equation}\label{a1}
				\displaystyle\int_{\mathbb{R}^N}a(x)\vert v_k\vert^qdx=\displaystyle\int_{\mathbb{R}^N\backslash B(0,R)} a(x)\vert v_k\vert^qdx+\int_{B(0,R)}a(x)\vert v_k\vert^qdx. 
			\end{equation}
			Since $v_k\rightharpoonup v$ in $H^1(\mathbb R ^N)$ there exists $h_l \in L^l(B(0,R)), 1\leq l<2^*$, such that $v_k \rightarrow v$ in $L^l(B(0,R))$. Moreover, we observe that $v_k(x) \rightarrow v(x)$ a. e. in $B(0,R)$ and $\vert v_k \vert \le h_l$ in $B(0,R)$. 	Notice also  that $\left\vert a(x)\vert v_k\vert^q\right\vert\leq a(x)h_l^q \in L^1(B(0,R))$. Hence, applying the Dominated Converge Theorem, we obtain that
			$$
			\displaystyle \lim_{k \to \infty} \int_{B(0,R)}a(x)\vert v_k\vert^qdx = \displaystyle\int_{B(0,R)}a(x)\vert v\vert^qdx.
			$$
			
			On the other hand, by using H\"older inequality and Remark \ref{imersaoRN}, we mention that 
			\begin{eqnarray*}
				\displaystyle\int_{\mathbb{R}^N\backslash B(0,R)}\!a(x)\vert v_k\vert^qdx 
				&\leq &\left[\left(\displaystyle\int_{\mathbb{R}^N\backslash B(0,R)}\!a(x)^rdx\right)^\frac{1}{r}\displaystyle\left(\int_{\mathbb{R}^N\backslash B(0,R)}\vert v_k\vert^{2^*}dx\right)^\frac{q}{2^*}\right] \leq S_{2^*}^{q}\left(\displaystyle\int_{\mathbb{R}^N\backslash B(0,R)}\!a(x)^rdx\right)^\frac{1}{r}\Vert v_k\Vert^q.
			\end{eqnarray*}
			Since $v_k\rightharpoonup v$ in $H^1(\mathbb R^N)$ there exists $C > 0$ in such way that $\Vert v_k\Vert\leq C$ holds. As a consequence, by using the Dominated Convergence Theorem, we obtain that 
			\begin{eqnarray*}
				\displaystyle\lim_{R\to \infty}\displaystyle\int_{\mathbb{R}^N\backslash B(0,R)}a(x)\vert v_k\vert^qdx 
				&\leq &  C_1 \displaystyle \lim_{R\to \infty}\displaystyle\left(\int_{\mathbb{R}^N\backslash B(0,R)}a(x)^rdx\right)^\frac{1}{r} = C_1 \lim_{R\rightarrow\infty}\displaystyle\left(\int_{\mathbb{R}^N}a(x)^r\chi_{\mathbb{R}^N\backslash B(0,R)}dx\right)^\frac{1}{r} = 0.
			\end{eqnarray*}
			Here was used the fact that $\vert a(x)^r\chi_{\mathbb R^N\backslash B(0,R)}\vert\leq a^r(x)\in L^1(\mathbb R^N)$. Similiarly, we also obtain that
			\begin{eqnarray*}
				\displaystyle\lim_{R\to \infty}\displaystyle\int_{\mathbb{R}^N\backslash B(0,R)}a(x)\vert v\vert^qdx 
				&\leq &  C_1 \lim_{R\rightarrow\infty}\displaystyle\left(\int_{\mathbb{R}^N}a(x)^r\chi_{\mathbb{R}^N\backslash B(0,R)}dx\right)^\frac{1}{r} = 0.
			\end{eqnarray*}
			Therefore, taking into account \eqref{a1} and the estimates given just above, we obtain that
			\begin{eqnarray*}
				\displaystyle\lim_{k\rightarrow\infty}\int_{\mathbb{R}^N}a(x)\vert v_k\vert^qdx 
				& = & \displaystyle\int_{\mathbb{R}^N\backslash B(0,R)} a(x)\vert v\vert^qdx+\int_{B(0,R)}a(x)\vert v\vert^qdx =\int_{\mathbb{R}^N}a(x)\vert v\vert^qdx.
			\end{eqnarray*}
			This ends the proof. 
		\end{proof}
	\end{prop}
	\begin{prop}\label{c2} 	Suppose ($H_1$)$-$($H_4$). Let $(v_k)\in H^1(\mathbb R ^N)$ be a bounded sequence such that $v_k\rightharpoonup v$ in $H^1(\mathbb R ^N$). Then we obtain that 
		$$
		\displaystyle\lim_{k\longrightarrow\infty}\int_{\mathbb{R}^N}\int_{\mathbb{R}^N}\frac{b(y)\vert u_k\vert^pb(x)\vert u_k\vert^p}{\vert x\vert^\alpha\vert x-y\vert^\mu\vert y\vert^\alpha}dxdy=\displaystyle\int_{\mathbb{R}^N}\int_{\mathbb{R}^N}\frac{b(y)\vert u\vert^pb(x)\vert u\vert^p}{\vert x\vert^\alpha\vert x-y\vert^\mu\vert y\vert^\alpha}dxdy
		$$
		\begin{proof}
			Firstly, by using the fact that $u_k\rightharpoonup u$ in $H^1(\mathbb R^N)$ there exists $h_l\in L^l(B(0,R)), 1\leq l<2^*$, such that
			$u_k \rightarrow u$ in $L^l(B(0,R))$ and $u_k \rightarrow u$ a.e in $B(0,R)$ with $\vert u_k \vert \le h_l$ in $B(0,R)$. Now, for each $R>0$, we mention that
			\begin{eqnarray}\label{J_3cont}
				\nonumber\displaystyle\lim_{k\rightarrow\infty}\displaystyle \int_{\mathbb R^N}\int_{\mathbb R^N}\frac{b(y)\vert u_k\vert^p b(x)\vert u_k\vert^p}{\vert x\vert^\alpha\vert x-y\vert^\mu\vert y\vert^\alpha}dxdy & = &\lim_{k\rightarrow\infty}\displaystyle \int_{ B(0,R)}\int_{B(0,R)}\frac{b(y)\vert u_k\vert^pb(x)\vert u_k\vert^p}{\vert x\vert^\alpha\vert x-y\vert^\mu\vert y\vert^\alpha}dxdy\\
				\nonumber & + & 2\lim_{k\rightarrow\infty}\displaystyle \int_{\mathbb R^N\backslash B(0,R)}\int_{B(0,R)}\frac{b(y)\vert u_k\vert^p b(x)\vert u_k\vert^p}{\vert x\vert^\alpha\vert x-y\vert^\mu\vert y\vert^\alpha}dxdy\\
				& + & \displaystyle\lim_{k\rightarrow\infty}\displaystyle \int_{\mathbb R^N\backslash B(0,R)}\int_{\mathbb R^N\backslash B(0,R)}\frac{b(y)\vert u_k\vert^pb(x)\vert u_k\vert^p}{\vert x\vert^\alpha\vert x-y\vert^\mu\vert y\vert^\alpha}dxdy.
			\end{eqnarray}
			The main idea here is to  analyze the term in the right hand side given in \eqref{J_3cont}. Initially, by using Proposition \ref{Stein-Weiss} and the same ideas discussed in the proof of Proposition \ref{Brezis-Lieb}, we observe that 
			$$\left\vert\frac{b(y)\vert u_k\vert^p b(x)\vert u_k\vert^p}{\vert x\vert^\alpha\vert x-y\vert^\mu\vert y\vert^\alpha}\right\vert\leq \frac{b(y)h_l^pb(x)h_l^p}{\vert x\vert^\alpha\vert x-y\vert^\mu\vert y\vert^\alpha}\in L^1(B(0,R)) \times  L^1(B(0,R)).$$
			Thus, using the Dominated Convergence Theorem, we deduce that
			\begin{equation}\label{dentrodab}
				\displaystyle \lim_{k\rightarrow \infty}\int_{B(0,R)}\int_{B(0,R)}\frac{b(y)\vert u_k\vert^pb(x)\vert u_k\vert^p}{\vert x\vert^\alpha\vert x-y\vert^\mu\vert y\vert^\alpha}dxdy
				=\displaystyle\int_{B(0,R)}\int_{B(0,R)}\frac{b(y)\vert u\vert^pb(x)\vert u\vert^p}{\vert x\vert^\alpha\vert x-y\vert^\mu\vert y\vert^\alpha}dxdy.
			\end{equation}
			Once again, by using and Proposition \ref{Stein-Weiss} and the H\"older inequality, we infer that 
			\begin{eqnarray*}\label{A1}
				&&\displaystyle\int_{\mathbb R^N\backslash B(0,R)}\int_{\mathbb R^N\backslash B(0,R)}\frac{b(y)\vert u_k-u\vert^p b(x)\vert u_k-u\vert^p}{\vert x\vert^\alpha\vert x-y\vert^\mu\vert y\vert^\alpha}dxdy \nonumber \\
				&\leq&  [\Vert b\Vert_{L^{\sigma}(\mathbb R^N\backslash B(0,R))}^{s_1}\Vert u_k-u\Vert^{s_1 p}_{2^*}]^{2/s_1} \leq  c \left(\displaystyle\int_{\mathbb R^N\backslash B(0,R)}\vert b\vert^{\sigma}dx\right)^{\frac{2}{\sigma}}.
			\end{eqnarray*}
			where $s_1 = 2N/(2N - 2 \alpha - \mu), s_2 = (2^*/s_1p)'$. Recall also that $\vert b^{\sigma}\chi_{\mathbb R^N\backslash B(0,R)}\vert\leq b^{\sigma}\in L^1(\mathbb R^N)$. Hence, we apply the Dominated Convergence Theorem showing that  
			\begin{eqnarray}\label{foradab}
				\nonumber & &\displaystyle\lim_{R\to \infty}\displaystyle\int_{\mathbb R^N\backslash B(0,R)}\int_{\mathbb R^N\backslash B(0,R)}\frac{b(y)\vert u_k-u\vert^p b(x)\vert u_k-u\vert^p}{\vert x\vert^\alpha\vert x-y\vert^\mu\vert y\vert^\alpha}dxdy \leq  c_2\lim_{R\rightarrow\infty} \left(\displaystyle\int_{\mathbb R^N}\vert b\vert^{\sigma}\chi_{\mathbb R^N\backslash B(0,R)}dx\right)^{\frac{2}{\sigma}}=0.
			\end{eqnarray}	
			It remains to prove that 
			\begin{eqnarray}\label{E21}
				\lim_{R\rightarrow\infty}\displaystyle \int_{\mathbb R^N\backslash B(0,R)}\int_{B(0,R)}\frac{b(y)\vert u_k\vert^p b(x)\vert u_k\vert^p}{\vert x\vert^\alpha\vert x-y\vert^\mu\vert y\vert^\alpha}dxdy  = \lim_{R\rightarrow\infty}\displaystyle \int_{\mathbb R^N\backslash B(0,R)}\int_{B(0,R)}\frac{b(y)\vert u\vert^p b(x)\vert u\vert^p}{\vert x\vert^\alpha\vert x-y\vert^\mu\vert y\vert^\alpha}dxdy = 0.
			\end{eqnarray}
			In order to do that, by using H\"older inequality, we deduce that 
			\begin{eqnarray}
				\nonumber&& \displaystyle\lim_{R\to\infty}\displaystyle \int_{B(0,R)}\int_{\mathbb R^N\backslash B(0,R)}\frac{b(y)\vert u_k-u\vert^p b(x)\vert u_k-u\vert^p}{\vert x\vert^\alpha\vert x-y\vert^\mu\vert y\vert^\alpha}dxdy\\
				\nonumber &\leq & \displaystyle\lim_{R\to \infty}\left[\displaystyle\int_{\mathbb R^N}\left(\displaystyle\int_{\mathbb R^N\backslash B(0,R)}\displaystyle\frac{b\vert u_k-u\vert^p}{\vert x\vert^{\alpha}\vert x-y\vert^{\mu}\vert y\vert^{\alpha}} dy \right)^{r_1} dx \right]^{\frac{1}{r_1}}\Vert b\vert u_k-u\vert^p\Vert_{L^{r_2}(B(0,R))}\\
				\nonumber &\leq &
				\displaystyle\lim_{R\to\infty}\left[\displaystyle\int_{\mathbb R^N}\left(\displaystyle\int_{\mathbb R^N}\displaystyle\frac{b\vert u_k-u\vert^p\chi_{\mathbb R^N\backslash B(0,R)}}{\vert x\vert^{\alpha}\vert x-y\vert^{\mu}\vert y\vert^{\alpha}} dy \right)^{r_1} dx \right]^{\frac{1}{r_1}}\Vert b\vert u_k-u\vert^p\Vert_{L^{r_2}(B(0,R))}.
			\end{eqnarray}
			Now, by using Proposition \ref{DHLSP1}, we obtain that
			\begin{eqnarray}
				&& \displaystyle\lim_{R\to\infty}\displaystyle \int_{B(0,R)}\int_{\mathbb R^N\backslash B(0,R)}\frac{b(y)\vert u_k-u\vert^p b(x)\vert u_k-u\vert^p}{\vert x\vert^\alpha\vert x-y\vert^\mu\vert y\vert^\alpha}dxdy \nonumber \\
				\nonumber &\leq &\displaystyle\lim_{R\to \infty}\left\Vert \displaystyle\int_{\mathbb R^N}\displaystyle\frac{b\vert u_k-u\vert^p\chi_{\mathbb R^N\backslash B(0,R)}}{\vert x\vert^{\alpha}\vert x-y\vert^{\mu}\vert y\vert^{\alpha}}\right\Vert_{r_1}\Vert b\vert u_k-u\vert^p\Vert_{L^{r_2}(B(0,R))}
				\nonumber \\ 
				& \leq &
				\displaystyle\lim_{R\to \infty} c\Vert b\vert u_k-u\vert^p\chi_{\R^N \backslash B(0,R)}\Vert_{r_3}\Vert b\vert u_k-u\vert^p\Vert_{L^{r_2}(B(0,R))} \nonumber 
			\end{eqnarray}
			where
			$
			1/r_2+ 1/r_3 +(2\alpha+\mu)/N=2, \,\,1/r_1+1/r_2=1.
			$
			Hence, by choosing $r_3=r_2$, we observe that $r_2= 2N/(2N-2\alpha-\mu)$. In particular,  by using the H\"older inequality, we deduce that 
			\begin{eqnarray}\label{E31}
				\nonumber\Vert b\vert u_k-u\vert^p\Vert^{r_2}_{L^{r_2}(B(0,R))}&=&\displaystyle\int_{B(0,R)}b^{r_2}(x)\vert u_k-u\vert^{pr_2}dx
				\leq \left(\displaystyle\int_{B(0,R)}b^{\sigma}dx\right)^{\frac{r_2}{\sigma}}\left(\displaystyle\int_{B(0,R)}\vert u_k-u\vert^{pr_2m_1}dx\right)^{\frac{pr_2}{pr_2m_1}}\\
				&\leq& \Vert b\Vert_{\sigma}^{r_2}\Vert u_k-u\Vert_{pr_2m_1}^{pr_2} \leq S_{2^*}^{pr_2}\Vert b\Vert_{\sigma}^{r_2}\Vert u_k-u\Vert^{pr_2}.
			\end{eqnarray}
			Here was used the fact that $m_1 > 1$ is the conjugate exponent of $\sigma/r_2$. Hence, $m_1=(2N-2\alpha-\mu)/p(N-2)$ and $pr_2m_1=2^*$ where $b\in L^{\sigma}(\mathbb R^N)$. Since $u_k\rightharpoonup u$ in $H^1(\mathbb R^N)$ we infer that $\Vert u_k-u\Vert\leq C$ holds for some $C > 0$. It follows from \eqref{E31} that
			\begin{eqnarray}\label{E311}
				\Vert b\vert u_k-u\vert^p\Vert^{r_2}_{L^{r_2}(B(0,R))}\leq CS_{2^*}^{pr_2}\Vert b\Vert_{\sigma}^{r_2}.
			\end{eqnarray}
			Once again, by using H\"older inequality, we infer that 
			\begin{eqnarray*}\label{F20}
				\nonumber\displaystyle\lim_{R\to \infty} c\Vert b\vert u_k-u\vert^p\chi_{\mathbb R^N \backslash B(0,R)}\Vert_{r_2} &=&\displaystyle\lim_{R\to \infty}\displaystyle\int_{\mathbb{R^N}}b^{r_2}(x)\chi_{\mathbb R^N\backslash B(0,R)}\vert u_k-u\vert^{pr_2}dx\\
				\nonumber&\leq & \displaystyle\lim_{R\to\infty}\left(\displaystyle\int_{\mathbb R^N}b(x)^{\sigma}\chi_{\mathbb R^N\backslash B(0,R)}dx\right)^{\frac{r_2}{\sigma}}\left(\displaystyle\int_{\mathbb R^N\backslash B(0,R)}\vert u_k-u\vert^{pr_2m_1}dx\right)^\frac{pr_2}{pr_2m_1}\\
				&\leq & \displaystyle S_{2^*}^{pr_2}\lim_{R\to\infty}\left(\displaystyle\int_{\mathbb R^N}b(x)^{\sigma}\chi_{\mathbb R^N\backslash B(0,R)}dx\right)^{\frac{r_2}{\sigma}} \Vert u_k-u\Vert^{pr_2}.
			\end{eqnarray*}
			Furthermore, we apply the Dominated Convergence Theorem showing that 
			\begin{eqnarray}\label{E30}
				&&\displaystyle\lim_{R\to \infty}\Vert b\vert u_k-u\vert^p\Vert^r_{L^{r_2}(\mathbb R^N\backslash B(0,R))}\leq C\displaystyle\lim_{R\to \infty}\displaystyle\int_{\mathbb R^N}b^{\sigma}\chi_{\mathbb R^N\backslash B(0,R)}(x)dx=0.
			\end{eqnarray}
			Here was used the fact that $\vert b^{\sigma}\chi_{\mathbb R^N\backslash B(0,R)}(x)\vert\leq b^{\sigma}\in L^1(\mathbb R ^N)$.  Therefore, taking into account \eqref{E311} and \eqref{E30}, we deduce that
			$$\lim_{R\to \infty}\displaystyle \int_{B(0,R)}\int_{\mathbb R^N\backslash B(0,R)}\frac{b(y)\vert u_k-u\vert^p b(x)\vert u_k-u\vert^p}{\vert x\vert^\alpha\vert x-y\vert^\mu\vert y\vert^\alpha}dxdy = 0.$$ 
			Hence, by using \eqref{J_3cont}, \eqref{dentrodab}, \eqref{foradab} and \eqref{E21}, we infer that
			\begin{equation*}\label{converg}
				\displaystyle\int_{\mathbb R^N}\int_{\mathbb R^N}\frac{b(y)\vert u_k \vert^p b(x)\vert u_k\vert^p}{\vert x\vert^\alpha\vert x-y\vert^{\mu}\vert y\vert^{\alpha}}dxdy
				\rightarrow\displaystyle\int_{\mathbb R^N}\int_{\mathbb R^N}\frac{b(y)\vert u \vert^p b(x)\vert u\vert^p}{\vert x\vert^\alpha\vert x-y\vert^{\mu}\vert y\vert^{\alpha}}dxdy.
			\end{equation*}
			This finishes the proof. 
		\end{proof}
	\end{prop}
	Now, we shall consider some useful properties for the functional $\Lambda_n$. Namely, we guarantee the following result:
	\begin{lemma}\label{0-homogeneo}
		Suppose ($H_1$)$-$($H_3$). 
		Consider the functional $\Lambda_n : H^1(\mathbb{R}^N) \to \mathbb{R}$ given by $\Lambda_n(u):=Q_n(t_n(u))$. Then we obtain the following statements:
		\begin{itemize}
			\item[i)] The functional $\Lambda_n$ is zero homogeneous, i.e., $\Lambda_n(tu)=\Lambda_n(u)$ for each $t>0, u\in H^1(\mathbb R^N)\setminus \{0\}$. Moreover, $\Lambda_n$ is a continuous and weakly lower semicontinuous;
			\item[ii)] There exists $v\in H^1(\mathbb R^N)\backslash \{0\}$ such that 
			$$
			\lambda^*=\Lambda_n(u)=\displaystyle\inf_{w \in H^1(\mathbb R^N)\backslash\{0\}}\Lambda_n(w).
			$$
			Hence, $\lambda^*$ is attained and $\lambda^*>0$.
			\item[iii)] Consider the function $v\in H^1(\mathbb R^N)\backslash \{0\}$ such that $\Lambda_n(v) = \lambda^*$. Define the function $w=t_n(v)v$. Then $w$ is a weak solution for the following elliptic problem
			\begin{eqnarray}
				\label{problemaextra} \ \ 
				\left\{\begin{array}{ll} 
					-2 \Delta w + 2w =  q\lambda^* a(x) |w|^{q-2} w + 2p\displaystyle \int_{\mathbb{R}^N}\displaystyle\frac{b(y)\vert w(y) \vert dy}{\vert x\vert^\alpha\vert x-y\vert^\mu \vert y\vert^\alpha} b(x)\vert w\vert^{p-2}w,\ \  \hbox{in}\ \mathbb{R}^N,
					\\ \\
					w\in H^1(\mathbb{R}^N).
				\end{array}\right.
			\end{eqnarray}
		\end{itemize}
	\end{lemma}
	\begin{proof}
		The proof of item $i)$ follows by using a straightforward computation. For the proof of item $ii)$ we observe that $\Lambda_n$ is bounded from below. In fact, by using Proposition \ref{Stein-Weiss} and Remark \ref{imersaoRN}, there exists $C > 0$ such that 
		\begin{equation*}
			\displaystyle\int_{\mathbb R^N}\int_{\mathbb R^N}\displaystyle\frac{b(y)\vert u(y)\vert^pb(x)\vert u(x)\vert^p}{\vert x\vert^{\alpha}\vert x-y\vert^{\mu}\vert y\vert^{\alpha}}dxdy\leq C \Vert b\Vert_{st}^{2}\Vert u\Vert^{2p}.
		\end{equation*}
		Under these conditions, we observe that 
		\begin{eqnarray*}
			\Lambda_n(u) & = & \displaystyle\frac{C_{p,q}\Vert u \Vert^{\frac{2p-q}{p-1}}}{A(u) B(u)^{\frac{2-q}{2p-2}}} \geq  \displaystyle\frac{C_{p,q}\Vert u \Vert^{\frac{2p-q}{p-1}}}{c_2\Vert a\Vert_r\Vert u\Vert^q\Vert b\Vert_{\sigma}^{\frac{2-q}{p-1}}\Vert u\Vert^{\frac{p(2-q)}{p-1}}}=C_3>0
		\end{eqnarray*}
		Therefore, $\Lambda_n(u)\geq C_3>0$.
		Now, we consider a minimizing sequence
		$(u_k)\in H^1(\mathbb R^N)\backslash\{0\}$, that is, $
		\Lambda_n(u_k)\to \lambda^*$.
		Now, we observe that  $\Lambda_n(u_k) \leq \lambda^* + 1$ holds for each $k \in \mathbb{N}$ large enough. In particular, we infer that 
		\begin{equation}\label{C5}
			\Lambda_n(u_k)=\frac{C_{p,q}\Vert u_k\Vert^{\frac{2p-q}{p-1}}}{A(u_k)\left[B(u_k)\right]^{\frac{2-q}{2p-2}}}< \lambda^* + 1.
		\end{equation}
		Consider the following auxiliary sequence
		$
		v_k= u_k/ A(u_k)^{\frac{1}{q}}.
		$ It is easy to verify that 
		$
		A(v_k)=1$ holds for each $k\in\mathbb N$.
		Since $\Lambda_n$ is zero homogeneous it follows that
		$\Lambda_n(v_k) = \Lambda_n(u_k) \to \lambda^*$. Therefore, by using \eqref{C5}, we also see that
		\begin{equation}\label{C6}
			\Vert v_k\Vert^{\frac{2p-q}{p-1}} \leq \widetilde{C}\left[B(v_k)\right]^{\frac{2-q}{2p-2}}.
		\end{equation}
		Now, by using \eqref{C6} and Remark \ref{imersaoRN}, we mention that
		\begin{equation}\label{achei}
			\Vert v_k\Vert^{\frac{2p-q}{p-1}} \leq \widetilde{C}_\epsilon\left[B(v_k)\right]^{\frac{2-q}{2p-2}}\leq c\Vert v_k\Vert^{\frac{p(2-q)}{p-1}}.
		\end{equation}
		Hence, by using the last estimate we see that 
		$
		\Vert v_k\Vert \leq c$ holds for some $c>0$.
		In particular, the sequence $(v_k)$ is bounded in $H^1(\mathbb R^N)$. Hence, there exists $v \in H^1(\mathbb R^N)$ such that  $v_k \rightharpoonup v$ in $H^1(\mathbb R^N)$. Therefore, there exists $h_l\in L^l(B(0, R)), 1\leq l<2^*$, such that
		$v_k \rightarrow v$ in $L^l(B(0,R))$ and
		$v_k \to v$ a. e. in $B(0,R)$. Moreover, we know that $\vert v_k \vert \le h_l$ in $B(0,R)$.
		Now, we claim that $v\neq 0$ is satisfied. The proof of this claim follows arguing by contradiction. Assuming that $v \equiv 0$ it follows from \eqref{C6} that $\|v_k\| \to 0$ as $k \to \infty$. However, by using \eqref{achei},  there exists $\delta > 0$ such that $0 < \delta \leq \|v_k\|$. This is a contradiction proving that $v \neq 0$. 
		Therefore, by using the fact that $\Lambda_n$ is weakly lower semicontinuous, we obtain that
		\begin{eqnarray*}
			0 < \lambda^* \leq \Lambda_n(v)\leq \displaystyle\lim_{k\to\infty}\inf\Lambda_n(v_k)=\lambda^*
		\end{eqnarray*}
		Therefore, $\Lambda_n(v)=\lambda^*>0
		$. This ends the proof of item $ii)$.
		
		Now, we shall prove the item $iii)$. Initially, using the fact that  $w = t_n(v)v \in H^1(\mathbb R^N)$ satisfies $\Lambda_n(w) = \lambda^*$, we mention that $t_n(v)$ is maximum point for the function $Q_n$ given by $Q_n(t):=R_n(tv), t > 0$. As a consequence, \begin{equation}\label{A}
			0=Q'_n(t_n(v))=R'_n(t_n(v)v)v.
		\end{equation}

		On the other hand, using the fact that $v$ is a critical point of $\Lambda_n$, we obtain that that
		\begin{eqnarray}\label{B}
			0 & = & (\Lambda_n(v))'\varphi =  (R_n(t_n(v)v))'\varphi = R'_n(t_n(v)v)[t'_n(v)\varphi]v+ R'_n(t_n(v)v)t_n(v)\varphi, \;\; \forall \varphi \in H^1(\mathbb R^N)
		\end{eqnarray}
		In light of \eqref{A} and \eqref{B} we deduce that
		$R_n'(t_n(v)v)\varphi=0$ holds for all $\varphi\in H^1(\mathbb R^N)$.
		In particular, for $w=t_n(v)v$, we obtain that 
		$$\lambda^*=R_n(w)=\frac{\Vert w\Vert^2-B(w)}{A(w)}.
		$$
		Under these conditions, we observe that 
		\begin{equation*}
			0 = R'_n(w)\varphi=\displaystyle\frac{1}{A(w)}\left[2\langle w,\varphi\rangle-2pD(w,\varphi)-q\lambda^*H(w,\varphi)\right].
		\end{equation*}
		Hence, $w$ is a weak solution to the Problem \eqref{problemaextra}. This finishes the proof.
	\end{proof}
	\begin{lemma}\label{lambda_* é atingido} Suppose ($H_1$)$-$($H_3$). 
		Let $\Lambda_e: H^1(\mathbb{R}^N \setminus \{0\}) \to \mathbb{R}$ given by $\Lambda_n(u):=Q_e(t_e(u))$. Then we obtain the following statements:
		\begin{itemize}
			\item [$i)$] The functional $\Lambda_e$ is zero homogeneous, i.e., $\Lambda_e(tu)=\Lambda_e(u)$ holds for each $t>0, u\in H^1(\mathbb R^N) \setminus \{0 \}$. Moreover, $\Lambda_e$ is a continuous and weakly lower semicontinuous.  
			\item[$ii)$] There exists $v\in H^1(\mathbb R^N)\backslash \{0\}$ such that
			$
			\lambda_*=\Lambda_e(v)=
			$
			Furthermore, $\lambda_*$ is attained and $\lambda_*>0$.
			\item[$iii)$] Define the function $w = t_e(u) u$. Then $w$ is a weak solution to the Problem \eqref{PEdcarlos}.
		\end{itemize}
	\end{lemma}
	\begin{proof}
		The proof follows the same ideas employed in the proof of Lemma \ref{0-homogeneo}. We omit the details. 
	\end{proof}
	
	Now, we shall prove that any nonzero function admits exactly one projection in the Nehari manifold $\mathcal{N}_\lambda^+$ and another projection in the Nehari manifold $\mathcal{N}_\lambda^-$. More specifically, we prove the following result:
	\begin{prop}\label{com 3 itens N^0 é vazia}
		Suppose ($H_1$)$-$($H_3$). Consider $\lambda \in (0,\lambda^*)$ and $u \in H^1(\mathbb R^N) \backslash\{0\}$. Then the fibering map $\phi(t) =J_{\lambda}(tu)$ has exactly two critical points $0<t_n^+(u)<t_n(u)<t_n^-(u)$. Furthermore, we consider the following statements:
		\begin{itemize}
			\item [$i)$] There holds $\mathcal{N}_{\lambda}^0 = \emptyset $ for each $\lambda \in (0,\lambda^*)$.
			\item[$ii)$] The number $t^+_n(u)$ is a local minimum point for the fibering map $\phi$ which satisfies $t^+_n(u)u \in \mathcal{N}_{\lambda}^+$. Furthermore, the number $t_n^-(u)$ is a local maximum for the fibering map $\phi$ which verifies $t_n^-(u)u \in \mathcal{N}_{\lambda}^-$.
			\item[$iii)$] The functional $u \to t_n^+(u)$ and $u \to t_n^-(u)$ belong to $C^1(H^1(\mathbb R^N)\backslash\{0\}, \mathbb R)$.
		\end{itemize}
		\begin{proof}
			$i)$ Assume that there exists $u\in \mathcal{N}_{\lambda}^0$. As a consequence, $t_n(u)=1$. The last assertion implies that
			$$\lambda<\lambda^*=\displaystyle\inf_{v \in H^1(\mathbb R^N)\backslash\{0\}}\Lambda_n(v)\leq \Lambda_n(u)=R_n(t_n(u)u)=R_n(u)=\lambda.$$
			This is a contradiction proving that  $\mathcal{N}_{\lambda}^0=\emptyset$ holds for each $\lambda \in (0, \lambda^*)$.
			
			$ii)$ Consider $u \in H^1(\mathbb R^N)\backslash\{0\}$ be fixed. In view of \eqref{lambda^*} we mention that $R_n(t_n(u)u)=Q_n(t_n(u)) \geq \lambda^*> \lambda $.
			Furthermore, by using Proposition \ref{relação entre R_n e J'}, we obtain that
			$R_n(tu)=\lambda$ if and only if $tu \in \mathcal{N}_\lambda$. As a consequence, the identity $R_n(t u) = \lambda$ admits exactly two roots in such way that $0 < t_n^+(u) < t_n(u) < t_n^-(u)$.
			Clearly, the roots $t_n^+(u)$ and $t_n^-(u)$ are critical points for the fibering map
			$
			\phi(t)=J_{\lambda}(tu).
			$
			Furthermore, by using Proposition \ref{d/dtR_n(tu) e J''}, we mention also that
			$$
			J_{\lambda}''(t_n^+(u)u)(t_n^+(u)u,t_n^+(u)u)>0 \,\, \mbox{and} \,\,
			J_{\lambda}''(t_n^-(u)u)(t_n^-(u)u,t_n^-(u)u)<0.
			$$
			Therefore, $t_n^+(u)u \in \mathcal{N}_{\lambda}^+$ 
			and $t_n^-(u)u \in \mathcal{N}_{\lambda}^-$ are satisfied. In particular, $t_n^-(u)$ is a local maximum point for $\phi$
			and $t_n^+(u)$ is a local minimum point for $\phi$. This ends the proof of item $ii)$.							
			
			$iii)$ Firstly, for each $\lambda \in (0,\lambda^*)$, we mention that 
			\begin{equation*}\label{lambda<R_n}
				\lambda<\lambda^*<Q_n(t_n(u))=R_n(t_n(u)u), \,\, u \in H^1(\mathbb R^N)\backslash\{0\}.
			\end{equation*}
			Notice also that $ \mathcal N_{\lambda} =\mathcal{N}_{\lambda}^+\cup\mathcal{N}_{\lambda}^-$ is satisfied for each $\lambda \in (0,\lambda^*)$. Hence, by using the Proposition \ref{d/dtR_n(tu) e J''} we deduce that 
			$R_n(t_n^+(u)u)=R_n(t_n^-(u)u)=\lambda.$ 
			Furthermore,  we mention that 
			$$\frac{d}{dt}R_n(t_n^+(u)u)>0\;\;\hbox{and}\;\;\frac{d}{dt}R_n(t_n^-(u)u)<0.$$
			Therefore, by using the Implicit Function Theorem \cite[Theorem 2.4.1]{DRABEK}, we obtain that
			$u \mapsto t_n^+(u)$ and $ u \mapsto t^-_n(v)$ are in $C^1(H^1(\mathbb R^N)\backslash\{0\}, \mathbb R)$. This ends the proof. 
		\end{proof}
		\begin{prop}\label{N0}
			Suppose ($H_1$)$-$($H_3$). Then $\mathcal N_{\lambda^*}^0\neq\emptyset$.
		\end{prop}
		\begin{proof}
			In view of Lemma \ref{0-homogeneo} there  exists $ u\in H^1(\mathbb R^N)\backslash\{0\}$ in such way that $$\lambda^*=\Lambda_n(u)=R_n(t_n(u)u), \;\;\;\;\frac{d}{dt} R_n(tu)\mid_{t=t_n(u)}=0.$$ Hence, by using Proposition \ref{d/dtR_n(tu) e J''}, we obtain that $J_{\lambda^*}''(t_n(u)u)(t_n(u)u,t_n(u)u)=0.$
			Hence, $t_n(u)u \in \mathcal N_{\lambda^*}^0,$
			which implies that $\mathcal N_{\lambda^*}^0\neq \emptyset$. This ends the proof. 
		\end{proof}
	\end{prop}
	
	\section{The Nehari method for the functional $J_\lambda$}
	
	In this section we shall prove some results for the functional $J_\lambda$ by applying the Nehari method. Firstly, we shall prove the following result:
	\begin{lemma}\label{coercivo}
		Suppose ($H_1$)$-$($H_3$). The energy functional $J_{\lambda}$ is coercive in the Nehari manifold $\mathcal N_{\lambda}$ for each $\lambda>0$. In particular, $J_{\lambda}$ is bounded from below in $\mathcal N_{\lambda}$.
		\begin{proof}
			Let $u \in \mathcal{N}_{\lambda}$ be a fixed function. Therefore, we obtain 
			$ B(u)=\Vert u\Vert^2-\lambda A(u).$
			Now, by using H\"older inequality and Remark \ref{imersaoRN} together with the last identity, we deduce that 
			\begin{eqnarray*}
				J_{\lambda}(u) 
				& = & \frac{1}{2}\left(1-\frac{1}{p}\right)\Vert u\Vert^2-\lambda\left[\frac{1}{q}-\frac{1}{2p}\right]A(u) \geq \frac{1}{2}\left(1-\frac{1}{p}\right)\Vert u\Vert^2-\lambda\left[\frac{1}{q}-\frac{1}{2p}\right]\Vert a\Vert_r\Vert u\Vert_2^q \nonumber \\
				&\geq&  c_1\Vert u \Vert^2-\lambda c_2\Vert u\Vert^q =  \Vert u\Vert^2\left(c_1-\lambda c_2\Vert u\Vert^{q-2}\right){\rightarrow}\infty,\;\;\Vert u\Vert\to +\infty, u \in \mathcal{N}_\lambda.
			\end{eqnarray*}
			Therefore, $J_{\lambda}$ is coercive and bounded from below in $\mathcal{N}_{\lambda}$. This ends the proof. 
		\end{proof}
	\end{lemma}
	\begin{lemma}\label{modulo u >c}
		Suppose ($H_1$)$-$($H_3$). Assume that $\lambda\in (0, \lambda^*].$ Then,
		for each $u \in \mathcal{N}_{\lambda}^-\cup\mathcal{N}_{\lambda}^0$ there is a constant $C=C(N,p,q)>0$ that does not depend on $\lambda$ in such a way that $\Vert u\Vert \geq C.$ In particular, the sets $\mathcal{N}_{\lambda}^-$ and $\mathcal{N}_{\lambda}^0$ are closed.
	\end{lemma}
	\begin{proof}
		Firstly, we assume that $\lambda \in (0,\lambda^*)$.
		Let $u \in \mathcal{N}_\lambda^-$ be a fixed function.
		Under these conditions, we deduce that  $1 = t_n^-(u)\geq t_n(u)$. In particular, by using Remark \ref{imersaoRN},  we obtain that
		\begin{eqnarray*}
			1  =  t_n^-(u) & \geq & t_n(u) =  \left[\frac{(2-q)\Vert u\Vert^2}{(2p-q)B(u)}\right]^{\frac{1}{2p-2}} \geq  \left[\frac{(2-q)\Vert u\Vert^2}{(2p-q)c\Vert b\Vert^2_{\sigma}\Vert u\Vert^{2p}}\right]^{\frac{1}{2p-2}} = c_1\Vert u\Vert^{-1}.
		\end{eqnarray*}
		Therefore,  $\Vert u\Vert \geq c_1(N,p,q)>0$ holds for each $u \in \mathcal{N}_\lambda^-$ whenever $\lambda \in (0, \lambda^*)$. Furthermore, by using a standard argument we obtain that $\mathcal{N}_\lambda^-$ is closed. 
		
		It is important to emphasize that $\mathcal{N}_{\lambda}^0=\emptyset$ for each $\lambda \in (0, \lambda^*)$, see Proposition \ref{com 3 itens N^0 é vazia}. 
		It remains to consider the case $\lambda=\lambda^*$. Notice also that $\mathcal{N}_{\lambda}^0\neq \emptyset$, see Proposition \ref{N0}.
		Consider $u \in \mathcal{N}_{\lambda}^0$ a fixed function. According to Proposition \ref{d/dtR_n(tu) e J''} we obtain that $R'_n(u)u=0$. Recall also that $u \in \mathcal{N}_{\lambda}^0$ implies that $\lambda^*=\Lambda_n(u)$.
		Since $t_n(u)$ is the unique maximum point of the function $Q_n$ it follows that 
		\begin{eqnarray*}
			1=t_n(u) & = & \left[\frac{(2-q)\Vert u\Vert^2}{(2p-q) B(u)}\right]^{\frac{1}{2p-2}} \geq  \left[\frac{(2-q)\Vert u\Vert^2}{(2p-q)c\Vert b\Vert^2_{\sigma}\Vert u\Vert^{2p}}\right]^{\frac{1}{2p-2}} =  c_2 \Vert u\Vert^{-1}.
		\end{eqnarray*}
		Therefore, $\Vert u\Vert \geq c_2(N,p,q)>0$ holds for each $u \in \mathcal{N}_{\lambda^*}^0$. It follows also that $\mathcal{N}_{\lambda}^0$ is closed. This finishes the proof. 
	\end{proof}
	For the next result we shall prove that $\mathcal{N}_{\lambda}$ is a natural constraint for our main problem for each $\lambda \in (0,\lambda^*)$. Indeed, we are in position to prove that any critical point for $J_{\lambda}$ in $\mathcal{N}_{\lambda}$ is a critical point for the functional $J_{\lambda}$. Hence, any minimizer $u\in \mathcal{N}_{\lambda}^+\cup\mathcal{N}_{\lambda}^-$ for the energy function $J_{\lambda}$ restricted to the Nehari manifold is a critical point for $J_{\lambda}$. More precisely, we show the following result:
	\begin{lemma} \label{u ponto critico}
		Suppose ($H_1$)$-$($H_3$). Assume that $u \in \mathcal{N}_{\lambda}^-\cup \mathcal{N}_{\lambda}^+$ is a local minimizer for $J_{\lambda}$ in $\mathcal{N}_{\lambda}$. Then $u$ is a critical point of $J_{\lambda}$ on $H^1(\mathbb R^N)$, that is, we obtain that $J_{\lambda}'(u)\varphi=0$ for each $\varphi \in H^1(\mathbb R^N)$ where $\lambda\in (0,\lambda^*).$
	\end{lemma}
	\begin{proof}
		Define the following function $\theta:H^1(R^N)\to \mathbb R$ given by $\theta(u)=J_{\lambda}'(u)u$. Notice that $\theta$ in $C^1$ class whose the Gauteaux derivative is given by
		\begin{equation*}\label{theta com J''}
			\theta'(u) w=J_{\lambda}''(u)(u,w) + J'_\lambda(u)w,  \,\, w \in H^1(\mathbb{R}^N).
		\end{equation*}
		Recall also that $\mathcal{N}_\lambda = \theta^{-1}(\{0\})$ and $\theta'(u) u \neq 0$ for each $u \in \mathcal{N}_\lambda$. Hence, applying the Implicit Function Theorem \cite{DRABEK}, we obtain that $\mathcal{N}_\lambda$ is a $C^1$ manifold. Therefore, by using the Lagrange Multiplier Theorem \cite{DRABEK}, there exists $\kappa \in \mathbb R$ such that  
		\begin{equation*}
			J_{\lambda}'(u)\varphi=\kappa \theta'(u)\varphi, \,\, \varphi \in H^1(\mathbb R^N).
		\end{equation*}
		Now, for $\varphi=u$ and using the fact that $u \in \mathcal{N}_{\lambda}$, we obtain that
		$
		0=J_{\lambda}'(u)u=\kappa \theta'(u)u.
		$
		Since $\theta'(u)u\neq 0$ holds true for each $u \in \mathcal{N}_\lambda$ we deduce that $\kappa=0$. Using the last assertion, we obtain that $J_{\lambda}'(u)\varphi=0$ holds for each $ \varphi \in H^1(\mathbb R^N)$. In particular, $u$ is a critical point of $J_{\lambda}$. This ends the proof. 
	\end{proof}
	For the next result we shall prove that any minimizer sequence in $\mathcal{N}_{\lambda}^-$ strongly converges. In fact, we prove the following result:
	\begin{lemma}\label{J(v)=C_N-}
		Suppose ($H_1$)$-$($H_3$). Assume that $\lambda\in(0,\lambda^*)$. Consider $(v_k)\subset \mathcal{N}_{\lambda}^-$ a minimizer sequence for the functional $J_{\lambda}$ in $\mathcal{N}_{\lambda}^-$. Then there exists $v \in H^1(\mathbb R^N)\backslash\{0\}$ such that, up to a subsequence, $v_k \to v$ in $H^1(\mathbb R^N )$. Furthermore, $v\in \mathcal{N}_{\lambda}^-$ and  $C_{\mathcal{N}_{\lambda}^-}=J_{\lambda}(v)$.
	\end{lemma}
	\begin{proof}
		Consider a minimizer sequence $(v_k)\subset \mathcal{N}_{\lambda}^-$. As $J_{\lambda}$ is coercive in the Nehari set $\mathcal{N}_\lambda$, we obtain that $(v_k)$ is a bounded sequence. In particular, there exists  $v \in H^1(\R^N)$ such that $v_k \rightharpoonup v$ in $H^1(\R^N)$. Hence, there exists $h_l$ in $L^{l}(B(0,R))$ such that $v_k \to v$ in $L^{l}(B(0,R)), v_k \to v$ a.e in $B(0,R)$ and $|v_k|  \leq h_l$ in $B(0,R)$. Here we claim that $v \neq 0$. The proof follows arguing by contradiction assuming that $v \equiv 0$. Now, by using the same ideas discussed in the proof of Proposition \ref{c1} it follows that 
		$H(v_k, v_k) \to 0 \,\, \mbox{and} \,\, D(v_k,v_k) \to 0, \,\, k \to \infty.$
		Now, by using the last assertion together with Lemma \ref{modulo u >c}, we obtain that $0 < c \leq \|v_k\| \to 0$ as $k \to \infty$. This does not make sense proving that $v \neq 0$.
		
		It remains to prove that $v_k \to v$ in $H^1(\mathbb{R}^N)$. The proof follows arguing by contradiction. Let us assume that $v_k \not\rightarrow v$. In particular, we see that $\Vert v \Vert<\displaystyle\lim_{n \to \infty}\inf \Vert v_k \Vert$. Since $v\neq 0$ there exists $t_k^-(v) > 0$ such that $t_k^-(v)v \in \mathcal{N}_{\lambda}^-$. Recall that the function $ t \mapsto J_{\lambda}(tv)$ is increasing for each $t \in [t_n^+(u), t_n^-(u)]$. Now, by using the fact that $v \to J_{\lambda}'(v)v$ weakly lower semicontinuous, we obtain that
		$$ 0=J_{\lambda}'(t_n^-(v)v)t_n^-(v)v<\displaystyle\liminf_{k\to\infty} J_{\lambda}'(t_n^-(v)v_k)t_n^-(v)v_k.$$
		As a consequence, $J_{\lambda}'(t_n^-(v)v_k)v_k>0$ is verified for each $k \in \mathbb{N}$ large enough. In particular, we obtain also that $t_n^-(v)\in (t_n^+(v_k),t_n^-(v_k))$. Using the last assertion together with the fact that $v \to J_{\lambda}(v)$ is weakly lower semicontinuous we deduce that
		$$C_{\mathcal{N}_{\lambda}^-}\leq J_{\lambda}(t_n^-(v)v)<\displaystyle \liminf_{k \to \infty} J_{\lambda}(t_n^-(v) v_k)\leq \displaystyle \liminf_{k \to \infty} J_{\lambda}(t_n^-(v_k)v_k)=\liminf_{k \to \infty} J_{\lambda}(v_k)=C_{\mathcal{N}_{\lambda}^-}.$$
		This is a contradiction proving that $v_k \to v$ in $H^ 1(R^N)$. The desired result follows from the fact that $J_\lambda$ is in $C^1$ class.
	\end{proof}
	\begin{lemma}\label{C_n+<0}
		Suppose ($H_1$)$-$($H_3$). Assume also that $\lambda\in (0,\lambda^*)$. Then we obtain that  $C_{\mathcal{N}_{\lambda}^+}=J_{\lambda}(u)<0$.
		\begin{proof}
			It follows from Proposition \ref{relação entre R_n e J'} and Proposition \ref{relação entre R_e e J} that $J_{\lambda}(t_n^+(u)u)<0$ holds for each $u \in \mathcal{N}_\lambda^+$. As a consequence, we obtain that $$C_{\mathcal {N}_{\lambda}^+}=\displaystyle\inf_{v\in\mathcal {N}_{\lambda}^+}J_{\lambda}(v)\leq J_{\lambda}(t_n^+(u)u)<0.$$
			In particular, we obtain that $C_{\mathcal {N}_{\lambda}^+}<0$. This ends the proof.
		\end{proof}
	\end{lemma}
	\begin{lemma}\label{ponto critico N+}
		Suppose ($H_1$)$-$($H_3$). Assume also that $\lambda \in (0,\lambda^*)$. Let  $(u_k)\subset\mathcal{N}_{\lambda}^+$ be a minimizer sequence for 
		$J_{\lambda}$ in $\mathcal{N}_{\lambda}^+$. Then there exists $u \in H^1(\mathbb R^N)\backslash\{0\}$ such that, up to a subsequence, $u_k \to u$ in $H^1(\mathbb R^N)$. Furthermore, we obtain that $u \in \mathcal{N}_\lambda^+$ and  $C_{\mathcal{N}_{\lambda}^+}=J_{\lambda}(u)$.
		\begin{proof}
			Consider a minimizer sequence $(u_k)\subset\mathcal{N}_{\lambda}^+$. Since $J_{\lambda}$ is coercive in the Nehari set $\mathcal{N}_{\lambda}$ it follows that $(u_k)$ is a bounded sequence. Furthermore, we mention that 
			\begin{eqnarray}\label{igNeh}
				0=J_{\lambda}'(u_k)u_k=\Vert u_k\Vert^2-\lambda A(u_k)-B(u_k).
			\end{eqnarray}		
			As a consequence, by using  \eqref{FE}, we mention that
			\begin{equation}\label{2*}
				\Vert u_k \Vert^2=2J_{\lambda}(u_k)+\frac{2\lambda}{q}A(u_k)+\frac{1}{p}B(u_k).
			\end{equation}
			Now, by applying  \eqref{2*} and \eqref{igNeh}, we observe that
			\begin{equation*}
				\lambda A(u_k)=\frac{-2qJ_{\lambda}(u_k)}{2-q}+\frac{q(p-1)}{p(2-q)}B(u_k).
			\end{equation*}
			In view of Propositions
			\ref{c1} and \ref{c2} we infer that
			\begin{eqnarray*}
				\lambda A(u) & = & \frac{q(p-1)}{p(2-q)}B(u)-\frac{2p}{2-q}C_{\mathcal{N}_{\lambda}^+} \geq \frac{-2pC_{\mathcal{N}_{\lambda}^+}}{2-q}>0
			\end{eqnarray*}
			Here was used the fact that $C_{\mathcal{N}_{\lambda}^+}<0$, see Lemma \ref{C_n+<0}. The last estimate implies that $u\neq 0$.
			
			From now on, we shall prove that
			$u_k \to u$ in $ H^1(\mathbb R^N)$. Once again the proof follows arguing by contradiction. Let us assume that $u_k \not\rightarrow u$. Therefore,
			$\Vert u\Vert<\displaystyle\lim_{k\to \infty}\inf\Vert u_k\Vert$. Consider the fibering map $\phi:[0,\infty) \to \mathbb R$ given by  $\phi(t)=J_{\lambda}(tu), t\geq 0$.
			According to Proposition \ref{com 3 itens N^0 é vazia}, there exists a unique $t_n^+(u)>0$ in such way that $t_n^+(u)u \in \mathcal{N}_{\lambda}^+$. Furthermore, we know that $\phi'(t_n^+(u))=J_{\lambda}'(t_n^+(u)u)u=0$.
			Notice also that $u_k\in \mathcal {N}_{\lambda}^+$ and t $\mapsto J_\lambda(t u_k)$ is increasing for each $t \in [t_n^+(u), t^-_n(u)]$. Hence, we deduce that
			\begin{eqnarray*}
				0 = t_n^+(u)\displaystyle\frac{d}{dt}J_{\lambda}(t_n^+(u)u)=J_{\lambda}'(t_n^+(u)u)t_n^+(u)u<\displaystyle\liminf_{k\to\infty} J_{\lambda}'(t^+_n(u)u_k)t^+_n(u)u_k.
			\end{eqnarray*}
			Under these conditions, we we obtain that
			$J_{\lambda}'(t^+_n(u)u_k)t^+_n(u)u_k>0$ holds for each $k\in \mathbb{N}$ large enough.
			Consequently, $t_n^+(u_k)<t_n^+(u)<t_n^-(u_k)$. Moreover, we observe that 
			$1 = t_n^+(u_k) < t_n^+(u)$. Therefore, by using  
			the last assertion, we deduce that
			\begin{equation*}
				C_{\mathcal {N}_{\lambda}^+}\leq J_{\lambda}(t_n^+(u)u)<J_{\lambda}(u)<\displaystyle\lim_{k\to \infty}J_{\lambda}(u_k)=C_{\mathcal {N}_{\lambda}^+}.
			\end{equation*}
			This is a contradiction proving that that $u_k\to u$ in $ H^1(\mathbb R^N)$. Since $J_{\lambda}$ is in $C^1$ class the desired result follows.
		\end{proof}
	\end{lemma}
	At this stage, we consider a regularity result for our main problem. This result is quite important in order to ensure that our main problem admits at least two positive solutions, see Proposition \ref{u e v positivos} ahead. More specifically, we assume that $1\leq q< 2$ or $2\leq q < 2^*$ where $2^* = 2N/(N -2)$. Hence, we can prove the following result:
	
	\begin{prop}\label{regular} Suppose ($H_1$), ($H_2$), ($H_3$). Assume also that ($H_4$) holds for each $1\leq q < 2$ while $a\in L^{2^*N/(2^*2-(q-2)N)}(\mathbb{R}^N)$ is verified for each $q\in[2,2^*)$. Let $u\in H^1(\mathbb R^N)$ be a weak solution for the Problem \eqref{PEdcarlos}.  Then we obtain that $u\in W^{2,\theta}_{loc}(\mathbb R^N)\cap C^{1,\sigma}_{loc}(\mathbb R^N)$ for all $\theta \in (1,\infty)$ and for some $\sigma \in (0,1)$.
		\begin{proof}
			Let $u \in H^1(\mathbb R^N)$ be a fixed weak solution for the Problem \eqref{PEdcarlos}. Consider the auxiliary function $g : \mathbb{R}^N \to \mathbb{R}$ given by 
			$$g(x)=\displaystyle\frac{h(x,u)}{1+\vert u\vert}, x \in \mathbb{R}^N.$$ 
			Here we emphasize that
			$$h(x,u)=\lambda a(x)\vert u\vert^{q-2}u+\displaystyle\int_{\mathbb R^N}\displaystyle\frac{b(y)\vert u\vert^p dy}{\vert x\vert^{\alpha}\vert x-y\vert^{\mu}\vert y\vert^{\alpha}}b(x)\vert u\vert^{p-2}u- V(x)u,\;\;x \in \mathbb R^N.$$
			Thus, $u$ is a weak solution for the elliptic problem
			$$
			-\Delta u=g(x)(1+\vert u\vert),\;\; x \in \mathbb R^N.
			$$
			It is not hard to verify that
			\begin{eqnarray}
				\nonumber\vert h(x,u)\vert  &=& \left\vert\lambda a(x)\vert u\vert^{q-2}u+\displaystyle\int_{\mathbb R^N}\displaystyle\frac{b(y)\vert u\vert^p dy}{\vert x\vert^{\alpha}\vert x-y\vert^{\mu}\vert y\vert^{\alpha}}b(x)\vert u\vert^{p-2}u- V(x) u\right\vert\\
				%\nonumber & \leq & \lambda a(x)\vert u\vert^{q-2}\vert u\vert+\displaystyle\int_{\mathbb R^N}\displaystyle\frac{b(y)\vert u\vert^p dy}{\vert x\vert^{\alpha}\vert x-y\vert^{\mu}\vert y\vert^{\alpha}}b(x)\vert u\vert^{p-2}\vert u\vert+\|V\|_{\infty}\vert u\vert\\
				\nonumber & \leq & \lambda a(x)\vert u\vert^{q-1}+\displaystyle\int_{\mathbb R^N}\displaystyle\frac{b(y)\vert u\vert^p dy}{\vert x\vert^{\alpha}\vert x-y\vert^{\mu}\vert y\vert^{\alpha}}b(x)\vert u\vert^{p-1}+ \|V\|_{\infty}\vert u\vert.
			\end{eqnarray}
			Now, we shall split the prof into some cases. In the first one we assume that 
			$1\leq q<2$ and $p \in [2, 2^*)$. As a consequence, we mention that
			\begin{eqnarray*}
				\vert g(x) \vert
				\nonumber & \leq & \displaystyle\frac{\lambda a(x)\vert u\vert^{q-1}+\displaystyle\int_{\mathbb R^N}\displaystyle\frac{b(y)\vert u\vert^p dy}{\vert x\vert^{\alpha}\vert x-y\vert^{\mu}\vert y\vert^{\alpha}}b(x)\vert u\vert^{p-1}+ \|V\|_{\infty}\vert u\vert}{1+\vert u \vert}\\
				%\nonumber & \leq & \displaystyle\frac{\lambda a(x)\left(1 + \vert u \vert\right)^{q-1}+\displaystyle\int_{\mathbb R^N}\displaystyle\frac{b(y)\vert u\vert^p dy}{\vert x\vert^{\alpha}\vert x-y\vert^{\mu}\vert y\vert^{\alpha}}b(x)\vert u\vert^{p-1}+ 1 + \|V\|_{\infty} \vert u\vert}{1+\vert u \vert}\\
				& \leq & \lambda a(x)+\displaystyle\int_{\mathbb R^N}\displaystyle\frac{b(y)\vert u\vert^p dy}{\vert x\vert^{\alpha}\vert x-y\vert^{\mu}\vert y\vert^{\alpha}}b(x)\vert u\vert^{p-2}+ \max(1, \|V\|_{\infty}).
			\end{eqnarray*}
			Now, assuming that $q,p \in [2, 2^*)$, we infer that 
			\begin{eqnarray*}
				\vert g(x) \vert \nonumber &\leq&  \displaystyle\frac{\lambda a(x)\vert u\vert^{q-2}\vert u \vert+\displaystyle\int_{\mathbb R^N}\displaystyle\frac{b(y)\vert u\vert^p dy}{\vert x\vert^{\alpha}\vert x-y\vert^{\mu}\vert y\vert^{\alpha}}b(x)\vert u\vert^{p-1}+ \|V\|_{\infty}\vert u\vert}{1+\vert u \vert}\\
				& \leq & \lambda a(x)\vert u\vert^{q-2}+\displaystyle\int_{\mathbb R^N}\displaystyle\frac{b(y)\vert u\vert^p dy}{\vert x\vert^{\alpha}\vert x-y\vert^{\mu}\vert y\vert^{\alpha}}b(x)\vert u\vert^{p-2}+ \max(\|V\|_{\infty}, 1).
			\end{eqnarray*}
			Therefore, we obtain the following estimates
			\begin{eqnarray*}\label{sistema11}
				\vert g(x)\vert^{\frac{N}{2}}\leq  C \left\{\begin{array}{ll}
					\left[ a(x)+\displaystyle\int_{\mathbb R^N}\displaystyle\frac{b(y)\vert u(y)\vert^pdy}{\vert x\vert^{\alpha}\vert x-y\vert^{\mu}\vert y\vert^{\alpha}} b(x) \vert u(x)\vert^{p-2}+1\right]^{\frac{N}{2}}, \;\;q\in[1,2),  p \in (2, 2^*),\\
					\left[ a(x)\vert u\vert^{q-2}+\displaystyle\int_{\mathbb R^N}\displaystyle\frac{b(y)\vert u(y)\vert^pdy}{\vert x\vert^{\alpha}\vert x-y\vert^{\mu}\vert y\vert^{\alpha}}b(x)\vert u(x)\vert^{p-1}+1\right]^{\frac{N}{2}},\;\;q \in [2,2^*), p \in [2, 2^*).
				\end{array}
				\right.
			\end{eqnarray*}

			From now on, the main idea is to apply the Brezis-Kato Theorem, see \cite[ Lemma B.3]{STRUWE}. In order to do that we need to check that $g\in L^{\frac{N}{2}}(B(0, R))$ holds for all $R > 0$.
			Firstly, for each  $p \in [2, 2^*)$, we need to ensure that 
			\begin{equation}\label{achei2}
				\displaystyle\int_{\mathbb R^N}\displaystyle\frac{b(y)\vert u(y)\vert^pdy}{\vert x\vert^{\alpha}\vert x-y\vert^{\mu}\vert y\vert^{\alpha}}b(x)\vert u(x)\vert^{p-2}\in L^{\frac{N}{2}}(B(0,R)).
			\end{equation}
			In fact, by applying H\"older inequality, we infer that
			\begin{eqnarray}\label{conv}
				\displaystyle\int_{B(0,R)}\left[\displaystyle\int_{\mathbb R^N}\displaystyle\frac{b(y)\vert u(y)\vert^pdy}{\vert x\vert^{\alpha}\vert x-y\vert^{\mu}\vert y\vert^{\alpha}}b(x)\vert u(x)\vert^{p-2}\right]^{\frac{N}{2}}dx
				\nonumber& = & \displaystyle\int_{B(0,R)}\left[\displaystyle\int_{\mathbb R^N}\displaystyle\frac{b(y)\vert u(y)\vert^pdy}{\vert x\vert^{\alpha}\vert x-y\vert^{\mu}\vert y\vert^{\alpha}}\right]^{\frac{N}{2}}b(x)^{\frac{N}{2}}\vert u(x)\vert^{(p-2)\frac{N}{2}}dx\\
				\nonumber& \leq & \left[\displaystyle\int_{\mathbb R^N}\left(\displaystyle\int_{\mathbb R^N}\displaystyle\frac{b(y)\vert u(y)\vert^pdy}{\vert x\vert^{\alpha}\vert x-y\vert^{\mu}\vert y\vert^{\alpha}}\right)^{\frac{N}{2}p_1}dx\right]^{\frac{1}{p_1}}\left(\displaystyle\int_{\mathbb R^N}b^{\frac{N}{2}p_2}(x)dx\right)^{\frac{1}{p_2}} \nonumber\\
				&\times&\left(\displaystyle\int_{\mathbb R^N}\vert u\vert^{(p-2)\frac{N}{2}p_3}dx\right)^{\frac{1}{p_3}}.\nonumber
			\end{eqnarray}
			Since $\beta> N/2$ we choose $p_2:=2\beta/N$. Under these conditions, we consider $$p_3=\displaystyle\frac{2.2^*}{N(p-2)} = \dfrac{4}{(N-2)(p-2)}, p_1 = \dfrac{2.2^* \beta}{2.2^* \beta - N(p -2)\beta - N 2^*} = \dfrac{4 \beta}{\beta[4 - (p-2)(N-2)] - 2N}.$$
			Recall also that $p_1, p_2, p_3>1$ is verified for each $\beta > 2N/[ 4 - (p -2)(N-2)]$. Thus, by using \eqref{conv} and Proposition \ref{imersaoRN}, we infer that
			\begin{eqnarray}\label{m13}
				\displaystyle\int_{B(0,R)}\left[\displaystyle\int_{\mathbb R^N}\displaystyle\frac{b(y)\vert u(y)\vert^pdy}{\vert x\vert^{\alpha}\vert x-y\vert^{\mu}\vert y\vert^{\alpha}}b(x)\vert u(x)\vert^{p-2}\right]^{\frac{N}{2}}dx & \leq & \left\Vert \displaystyle\int_{\mathbb R^N}\frac{b(y)\vert u(y)\vert^p}{\vert x\vert^{\alpha}\vert x-y\vert^{\mu}\vert y\vert^{\alpha}}dy\right\Vert^{\frac{N}{2}}_{\frac{Np_1}{2}} \Vert b\Vert^{\frac{N}{2}}_{\beta}\Vert u \Vert ^{\frac{N}{2}(p-2)}_{2^*}\nonumber\\
				& \leq & C \left\Vert \displaystyle\int_{\mathbb R^N}\frac{b(y)\vert u(y)\vert^p}{\vert x\vert^{\alpha}\vert x-y\vert^{\mu}\vert y\vert^{\alpha}}dy\right\Vert^{\frac{N}{2}}_{\frac{Np_1}{2}}\Vert b\Vert^{\frac{N}{2}}_{\beta}\Vert u\Vert^{\frac{N}{2}(p-2)}.
			\end{eqnarray}
			On the other hand, by using Proposition \ref{DHLSP1}, we deduce that 
			\begin{equation}\label{m10}
				\left\Vert \displaystyle\int_{\mathbb R^N}\frac{b(y)\vert u(y)\vert^p}{\vert x\vert^{\alpha}\vert x-y\vert^{\mu}\vert y\vert^{\alpha}}dy\right\Vert_{\frac{Np_1}{2}}\leq c\Vert b\vert u\vert^p\Vert_{t_1},
			\end{equation}
			where $r_1'= Np_1/2$ and $1/r_1 + 1/t_1 + (2\alpha + \mu)/N = 2$.
			In particular, we obtain that 
			$$t_1'= \dfrac{Np_1}{(2 \alpha + \mu)p_1 - 2N},~t_1 =\frac{Np_1}{p_1( N - 2 \alpha - \mu) + 2} = \displaystyle\frac{2N\beta}{\beta[4-(N-2)(p-2)+2(N-2\alpha-\mu)]-2N}.$$ 
			It is important to emphasize that $t_1 > 1$ holds true for each $$p > \frac{2(N - 2 \alpha - \mu)}{N-2} - \frac{2^*}{\beta}.$$
			Furthermore, by using \eqref{m10} and applying H\"older inequality, we obtain that
			\begin{equation}\label{X1}
				\Vert b\vert u\vert^p\Vert^{t_1}_{t_1}\leq \left(\int_{\mathbb R^N}b^{t_1t_2}dx\right)^{\frac{1}{t_2}}\left(\int_{\mathbb R^N}\vert u\vert^{pt_1t_3} dx\right)^{\frac{1}{t_3}}.
			\end{equation}
			Now, we consider
			$$t_3=\displaystyle\frac{2^*}{pt_1}=\displaystyle\frac{\beta[4-(N-2)(p-2)+2(N-2\alpha-\mu)]-2N}{p\beta(N-2)}.$$
			Notice that $t_3>1$ is satisfied for each $\beta > 2N/[ 4 - (p -2)(N-2)]$. Hence, $t_2 > 1$ is the conjugate exponent of $t_3$ which is given by 
			$$t_2=\displaystyle\frac{\beta[4-(N-2)(p-2)+2(N-2\alpha-\mu)]-2N}{\beta[4-(N-2)(2p-2)+2(N-2\alpha-\mu)]-2N}.$$
			Hence, we obtain that 
			$$t_1t_2=\displaystyle\frac{N\beta}{\beta[2-(N-2)(p-1)+(N-2\alpha-\mu)]-N}.$$
			In view of hypothesis ($H_4$) we obtain that
			$b\in L^{t_1t_2}(\R^N)$ and $\beta[2-(N-2)(p-1)+(N-2\alpha-\mu)]-N>0$.  Therefore, by using \eqref{X1} and Proposition \ref{imersaoRN}, we deduce that
			$$\Vert b\vert u\vert^p\Vert^{t_1}_{t_1}\leq \Vert b\Vert_{t_1t_2}^{t_1}\Vert u\Vert_{2^*}^{pt_1} \leq C \Vert b\Vert_{t_1t_2}^{t_1}\Vert u\Vert^{pt_1}<\infty.$$
			Now, by using \eqref{m13} and \eqref{m10}, we obtain that \eqref{achei2} is now verified. 
			
			In view of hypothesis $(H_4)$ we observe that  $a\in L^{\frac{N}{2}}(B(0,R))$ which is sufficient for each $1 \leq q < 2$.
			It remains to show that
			$\lambda a(x)\vert u \vert ^{q-2} \in L^{\frac{N}{2}}(B(0,R))$ holds for each $q\in (2,2^*)$. In light of H\"older inequality and Remark \ref{imersaoRN}, we observe that
			\begin{eqnarray}
				\int_{B(0,R)}( a(x)\vert u\vert^{q-2})^{\frac{N}{2}}dx &=&\displaystyle\int_{B(0,R)}a(x)^{\frac{N}{2}}\vert u \vert^{(q-2)\frac{N}{2}}dx \nonumber \\
				& \leq& \left(\displaystyle\int_{B(0,R)}a(x)^{N\ell/2} dx\right)^{1/\ell}\left(\displaystyle\int_{B(0,R)}\vert u \vert^{2^*} dx\right)^{(q-2)N/2.2^*}\nonumber\\
				\nonumber&\leq&
				\Vert a\Vert^{\frac{N}{2}}_{N \ell/2} \Vert u\Vert^{\frac{N(q-2)}{2}}_{2^*}\leq C 	 \Vert a\Vert^{\frac{N}{2}}_{N \ell/2} \Vert u\Vert^{\frac{N(q-2)}{2}}<\infty.
			\end{eqnarray}
			It is important to mention that $$ \ell = \frac{2.2^*}{2.2^* - (q -2)N}, 2 < q < 2^*.
			$$ 
			In particular, we see that $\lambda a(x)\vert u\vert^{q-2}\in L^{\frac{N}{2}}(B(0,R))$ holds for each $q \in (2,2^*]$.
			Under these conditions, we conclude that the function $g \in L^{\frac{N}{2}}(B(0,R))$. Therefore, by using \cite[Lemma B3]{STRUWE}, we obtain that $u \in L^{\theta}(B(0,R))$ hods for each $1<\theta<\infty$. Now, by using the Calderón-Zygmund inequality \cite[Theorem 9.9]{TRUDINGER}, we obtain also that $u \in W^{2,\theta}(B(0,R))\hookrightarrow C^{0,\sigma}(\overline{B(0,R)})$. Furthermore, by choosing $\theta> N/2$, there exists $\sigma \in (0,1)$ such that that $u \in C^{1,\sigma}(\overline{B(0,R)})$. Hence, $u \in W_{loc}^{2,\theta}(\mathbb R^N)\cap C^{1, \sigma}(\overline{B(0,R)})$ holds for each $R  > 0$. This ends the proof.
		\end{proof}
	\end{prop}
	Now, we are in a position to show that the critical points of the energy functional $J_\lambda$ can be chosen  strictly positive in $\mathbb{R}^N$. More precisely, we have the following result:
	\begin{prop}\label{u e v positivos}
		Suppose ($H_1$), ($H_2$), ($H_3$), ($H_4$). Then the energy functional $J_{\lambda}$ admits at least two critical points $u$ and $v$ which are strictly positive in $\mathbb R^N$ for each $\lambda\in(0,\lambda^*)$.
		\begin{proof}
			According to Lemma \ref{J(v)=C_N-} there exists $v\in \mathcal{N}^-_{\lambda}$ such that
			$
			C_{\mathcal {N}_{\lambda}^-} = J_{\lambda}(v)$ holds for each  $\lambda\in(0,\lambda^*).
			$
			Furthermore, by using Lemma \ref{u ponto critico}, we guarantee that the minimizer $v\in\mathcal {N}_{\lambda}^-$ is a critical point of $J_{\lambda}$. Thus, $v$ is a weak solution to Problem \eqref{PEdcarlos}. Now, by using \eqref{FE} and \eqref{J'}, we deduce that $J_{\lambda}(v)=J(\vert v\vert)$ and $J_{\lambda}'(\vert v\vert)\vert v \vert=J_{\lambda}'(v)v=0$. Hence, $\vert v\vert \in \mathcal {N}_{\lambda}$. Moreover, by using \eqref{J''}, we observe that $J_{\lambda}''(\vert v\vert)(\vert v\vert,\vert v\vert)=J_{\lambda}''(v)(v,v)<0$. It follows that $\vert v\vert\in\mathcal {N}_{\lambda}^-$. Under these conditions, we obtain that $\vert v \vert\in \mathcal {N}_{\lambda}^-$ is a minimizer for $J_{\lambda}$ in $\mathcal {N}_{\lambda}$. Therefore, by using Lemma \ref{u ponto critico}, we deduce that $\vert v\vert$ is a critical point of functional energy $J_{\lambda}$. The last assertion implies that $\vert v \vert$ is a weak solution to Problem \eqref{PEdcarlos}. Without loss of generality, we assume that $v\geq 0$ in $\mathbb R^N$. Hence, the energy functional $J_{\lambda}$ admits at least one critical point $v\in H^1(\mathbb R^N)$ such that $v\geq 0$ in $\mathbb R^N$ holds for each $\lambda\in (0,\lambda^*)$. Now, we observe that $v \in C_{loc}^{1,\sigma}(\mathbb R^N)$, see Proposition \ref{regular}. 
			
			In what follows, the main objective is to ensure that $v>0$ in $\mathbb{R}^N$. The proof follows arguing by contradiction. Assume that there exists  $x_0 \in \mathbb R^N$ such that $v(x_0)=0$. Let $R>0$ be fixed. Recall also that $v\in C^{1,\sigma}(B(x_0, R))$ holds for some $\sigma \in (0,1)$. Therefore, $v$ satisfies the following inequalities
			\begin{eqnarray*}\label{sistema5}
				\left\{\begin{array}{ll}
					-\Delta v+v\geq 0,\;\;\hbox{in}\;\; B(x_0,R), \\
					v\geq 0,\;\;\hbox{in}\;\;\partial B(x_0,R).
				\end{array}
				\right.
			\end{eqnarray*}
			Now, by using the Strong Maximum Principle \cite{Carlos}, we obtain that $v=0$ in $B(x_0, R)$ or $v>0$ in $B(x_0, R)$. 
			Assuming that $v=0$ in $B(x_0,R)$ and taking into account that $R>0$ is arbitrary, we guarantee that $v=0$ in $\mathbb R^N$. However, we know that $0 < \delta \leq \|v\|$ holds for each $v \in  \mathcal {N}_{\lambda}^-$. This is a contradiction proving that $v>0$ in $\mathbb R^N$.
			
			Now, we shall prove that there exists another solution which is strictly positive in $\R^N$. Firstly, by using Lemma \ref{ponto critico N+}, there exists $u\in \mathcal{N}^+_{\lambda}$ such that
			$
			C_{\mathcal {N}_{\lambda}^+}=J_{\lambda}(u).
			$
			Furthermore, we obtain that $u$ is the critical point of $J_{\lambda}$ for each $\lambda \in (0, \lambda^*).$ Now, by using Lemma \ref{C_n+<0}, we obtain that $J_{\lambda}(u)<0$ is satisfied for each $\lambda \in (0, \lambda^*)$. The last statement implies that $u\neq 0$. Now, using the same ideas discussed just above, we can verify that $u>0$ in $\mathbb R^N$. Recall also that $\mathcal {N}_{\lambda}^-\cap\mathcal {N}_{\lambda}^+=\emptyset$. As a consequence, Problem \eqref{PEdcarlos} admits at least two positive solutions for each $\lambda \in (0, \lambda^*)$. This finishes the proof.
		\end{proof}
	\end{prop}
	\begin{prop}\label{w minimizador de Lambda _e}
		Suppose ($H_1$), ($H_2$), ($H_3$), ($H_4$). Assume also that $\lambda=\lambda_*$. Then the energy functional $J_{\lambda}$ admits a critical point $w\in H^1(\mathbb R^N)\backslash \{0\}$ such that $w$ is a minimizer for the functional $\Lambda_e$.
		\begin{proof}
			Recall that $\lambda_*$  is attained, see Lemma \ref{lambda_* é atingido}. As consequence, there exist $u \in H^1(\mathbb{R}^N) \setminus \{0\}$ such that
			$
			\lambda_*=\Lambda_e(u)=R_e(t_e(u)u).
			$
			Notice also that $t_e(u)>0$ is the maximum point of $Q_e(t)=R_e(tu), t > 0$. Let us consider the function $w=t_e(u)u$. Therefore,
			\begin{equation}\label{5*}
				0=Q'_e(1)=(R_e)'(w)w.
			\end{equation}
			On the other hand, using the fact that $u$ is the critical point of the functional $\Lambda_e$, we infer that
			$
			0 = \Lambda'_e(w)\varphi =[R_e(w)]'\varphi.
			$
			Furthermore, by using \eqref{5*}, we obtain that $R_e'(w)\varphi=0$ holds for each $\varphi\in H^1(\mathbb R^N)$. Now, applying \eqref{D4} and  Lemma \ref{lambda_* é atingido}, we obtain the following identity
			$$
			0=R'_e(w)\varphi=\frac{qJ_{\lambda_*}'(w)\varphi}{A(u) }, \, \, \varphi \in H^1(\mathbb R^N).
			$$
			Hence, $J_{\lambda_*}'(w)\varphi=0$ holds for each $\varphi \in H^1(\mathbb R^N)$. This finishes the proof. 
		\end{proof}
	\end{prop}
	
	\section{The proof of the Theorem \ref{teorema}}
	Firstly, the proof of existence of two minimizers in the Nehari manifolds $\mathcal{N}_\lambda^-$ and $\mathcal{N}_\lambda^+$  follows by using Lemma \ref{J(v)=C_N-} and Lemma \ref{ponto critico N+}, respectively. According to Proposition \ref{regular} we can choose $u > 0$  and $v > 0$ in $\mathbb{R}^N$. Furthermore, we know that $J_\lambda(u) < 0$, see Proposition \ref{C_n+<0}. Hence our main result follows depending on the sign of $J_{\lambda}(v)$ using the size of $\lambda>0$. More precisely, we consider the following items:
	\subsection{The sign of $J_\lambda(v)$ for $\lambda\in(0,\lambda_*)$}
	Firstly, we observe that $v\in\mathcal{N}_{\lambda}^-$ and $t_n^-(v)=1$. In particular, we obtain that $t_e(v)<t_n^-(v)$. Now, by using Proposition \ref{relação entre Q_n e Q_e}, we see that $Q_n(t_n^-(v))<Q_e(t_n^-(v))$. As a consequence, we infer that $\lambda=R_n(t_n^-(v)v)<R_e(t_n^-(v)v)=R_e(v)$. Therefore, by using  Proposition \ref{relação entre R_e e J}, we obtain that $J_{\lambda}(v)>0$.

	\subsection{The sign of $J_\lambda(v)$ for $\lambda = \lambda_*$}
	According to Proposition \ref{w minimizador de Lambda _e} there exists $w \in H^1(\mathbb{R}^N)$ which is a minimizer for the functional $\Lambda_e$, that is, $\lambda_*=\Lambda_e(w)$. Furthermore, $w$ is a critical point of the energy functional $J_{\lambda}$ where $\lambda = \lambda^*$. In particular, $J_{\lambda}'(w)\psi=0$ holds for all $\psi \in H^1(\mathbb R^N)$. Hence, for $\psi=w$, we obtain that $J_{\lambda}'(w)w=0$. The last assertion implies that $w\in\mathcal{N}_{\lambda}$. Now, using the same ideas discussed just before, we obtain that $w=t_e(u)u$, see Proposition \ref{w minimizador de Lambda _e}. In particular, $t_e(u)=t_n^-(u)$. As a consequence, we obtain the following identity
	$$
	\lambda=\lambda_*=\Lambda_e(u)=R_e(t_e(u)u)=R_e(t_n^-(u)u).
	$$
	Now, by using Proposition \ref{relação entre R_e e J}, we infer that $J_{\lambda}(t^-_n(u)u)=0$. Let us consider any minimizer $v\in\mathcal{N}_{\lambda}^-$. Hence,  $J_{\lambda}(v)=C_{\mathcal{N}_{\lambda}^-}\leq J_{\lambda}(t^-_n(u)u)=0$. The last assertion implies that $J_{\lambda}(v)\leq 0$ is now verified. It is not hard to see that $t_e(v)\leq t_n^-(v)=1$ for all $v \in \mathcal{N}_{\lambda}^-$. In particular, by using the fact that $R_n(tv)\leq R_e(tv), t\geq t_e(u)$, we see that $$\lambda=R_n(v)=R_n(t_n^-(v)v)\leq R_e(t_n^-(v)v)=R_e(v).$$ Thus, $\lambda\leq R_e(v)$ is satisfied. Once again, by using the Proposition \ref{relação entre R_e e J}, we obtain that $J_{\lambda}(v)\geq 0$.
	Under these conditions, we deduce that $J_{\lambda}(v)=0$ whenever $\lambda=\lambda_*$.

	\subsection{The sign of $J_\lambda(v)$ for $\lambda\in(\lambda_*, \lambda^*)$}
	Notice that $v \in \mathcal{N}_\lambda^-$ which implies that $t_n^-(v) = 1$. It is not hard to verify that $t^-_n(v) < t_e(v)$ with $\lambda \in (\lambda_*, \lambda^*)$. As a consequence, by using Proposition \ref{relação entre Q_n e Q_e}, we see also that $\lambda =R_n(t_n^-(v)v) > R_n(t_e(v)v) = R_e(t_e(v)v) > R_e (v)$. 
	Now, by using the Proposition \ref{relação entre R_e e J}, we see that $J_{\lambda}(v)=C_{\mathcal{N}_{\lambda}^-} < J_{R_e(v)}(v)=0$.
	In particular, we obtain that $J_{\lambda}(v)<0$ for each $\lambda\in(\lambda_*, \lambda^*)$. This ends the proof.

	%\begin{acknowledgement}
	%	The authors would like to thank the referees for careful reading the manuscript and valuable comments and suggestions.
	%\end{acknowledgement}

	%%%%%%%%%%%%%%%%%%%%%%%%%%%%%%%%%%%%%%%%%%%%%%%%%%%%%%%%%%%%%%%%%%%%%%%%%%%%%%%%%%%%%%%%%%%%%%%%%%%%%%%%%%%%%%%%%%%%%%%%%%%%%%%%%%%%%%%%%%%%%%%%%%%%%%%%%%%%%%%%%%%%%%%%%%%%%%%%%%%%%%%
	%                                                                             REFERENCES
	%%%%%%%%%%%%%%%%%%%%%%%%%%%%%%%%%%%%%%%%%%%%%%%%%%%%%%%%%%%%%%%%%%%%%%%%%%%%%%%%%%%%%%%%%%%%%%%%%%%%%%%%%%%%%%%%%%%%%%%%%%%%%%%%%%%%%%%%%%%%%%%%%%%%%%%%%%%%%%%%%%%%%%%%%%%%%%%%%%%%%%%

\end{document}